\documentclass[a4paper]{amsart}%
\usepackage{amsfonts}
\usepackage{amsmath}
\usepackage{amssymb}
\usepackage{graphicx}%
\providecommand{\U}[1]{\protect\rule{.1in}{.1in}}
\newtheorem{theorem}{Theorem}
\theoremstyle{plain}

\newtheorem{corollary}{Corollary}

\newtheorem{lemma}{Lemma}

\newtheorem{proposition}{Proposition}
\newtheorem{remark}{Remark}

\numberwithin{equation}{section}
\begin{document}
\title{Krein's strings whose spectral functions are of polynomial growth}
\author{S.Kotani}
\address{Kwansei Gakuin University}
\email{kotani@kwansei.ac.jp}

\begin{abstract}
In the case of Krein's strings with spectral functions of polynomial growth a
necessary and sufficient condition for the Krein's correspondence to be
continuous is given.

\end{abstract}
\maketitle

\section{Introduction}

Let $\mathcal{M}$ be the totality of non-decreasing, right continuous
functions on $[0,\infty)$ satisfying%
\[
m(0-)=0,\text{ \ }m(x)\leq\infty,
\]
and set%
\[
\left\{
\begin{array}
[c]{l}%
l=\inf\left\{  x\geq0;\text{ }m(x)=\infty\right\}  ,\smallskip\\
a=\inf\left\{  x\geq0;\text{ }m(x)>0\right\}  .
\end{array}
\right.
\]
For $m\in\mathcal{M}$ denote $\varphi_{\lambda}\left(  x\right)
,\psi_{\lambda}\left(  x\right)  $ the solutions to%
\[
\left\{
\begin{array}
[c]{l}%
\varphi_{\lambda}\left(  x\right)  =1-\lambda%
{\displaystyle\int\nolimits_{0}^{x}}
\left(  x-y\right)  \varphi_{\lambda}\left(  y\right)  dm(y),\smallskip\\
\psi_{\lambda}\left(  x\right)  =x-\lambda%
{\displaystyle\int\nolimits_{0}^{x}}
\left(  x-y\right)  \psi_{\lambda}\left(  y\right)  dm(y),
\end{array}
\right.
\]
and define%
\[
h\left(  \lambda\right)  =\lim_{x\rightarrow l}\frac{\psi_{\lambda}\left(
x\right)  }{\varphi_{\lambda}\left(  x\right)  }=\int\nolimits_{0}^{l}%
\varphi_{\lambda}\left(  x\right)  ^{-2}dx.
\]
Then it is known that there exists a unique measure $\sigma$ on $[0,\infty)$
satisfying%
\[
h\left(  \lambda\right)  =a+\int\nolimits_{0}^{\infty}\frac{1}{\xi-\lambda
}d\sigma\left(  \xi\right)  ,
\]
and conversely, $h$ determines $m$ uniquely. Conventionally it is understood
that for $m\in\mathcal{M}$ taking $\infty$ identically on $[0,\infty)$ the $h$
vanishing identically corresponds, and for $m\in\mathcal{M}$ vanishing
identically on $[0,\infty)$ the $h$ taking identically $\infty$ corresponds.
This is the theorem obtained by Krein\cite{kr1} and $m$ is called Krein's
(regular) string. Later Kasahara\cite{ka} established the continuity for the
correspondence and applied it to show limit theorems for 1D diffusion
processes with $m$ their speed measures. Recently Kotani\cite{k3} extended
Kasahara's result to a certain kind of singular strings $m,$ namely to $m$
which is a non-decreasing and right continuous function on $\left(
-\infty,\infty\right)  $ satisfying%
\[
m(-\infty)=0,\text{ \ }m(x)\leq\infty,
\]
and%
\begin{equation}
\int\nolimits_{-\infty}^{a}x^{2}dm(x)<\infty\label{14}%
\end{equation}
for some $a$. When the condition (\ref{14}) is satisfied, the boundary
$-\infty$ is called as the limit circle type for the associated generalized
second order differential operator $d^{2}/dmdx.$ In this case he introduced a
new $h$ by%
\[
h\left(  \lambda\right)  =\lim_{x\rightarrow-\infty}\left(  x+\varphi
_{\lambda}(x)%
{\displaystyle\int\nolimits_{x}^{l}}
\frac{dy}{\varphi_{\lambda}(y)^{2}}\right)  =a+\int\nolimits_{0}^{\infty
}\left(  \frac{1}{\xi-\lambda}-\frac{\xi}{\xi^{2}+1}\right)  d\sigma\left(
\xi\right)  ,
\]
which satisfies%
\[
h^{\prime}\left(  \lambda\right)  =%
{\displaystyle\int\nolimits_{-\infty}^{l}}
\frac{\partial}{\partial\lambda}\varphi_{\lambda}\left(  x\right)  ^{-2}dx,
\]
and proved the continuity of the correspondence between $m$ and $h.$
Probabilistic applications of this result were given by
Kasahara-Watanabe\cite{k-w,k-w2} and it was interpreted from the point of view
of the excursion theory by Yano\cite{yano}. In this article we consider $m$
satisfying a milder condition than (\ref{14}), namely%
\[
\int\nolimits_{-\infty}^{a}\left\vert x\right\vert dm(x)<\infty,
\]
and obtain the continuity result under additional conditions on $m,$ which
allows any power growth of the spectral measures at $\infty.$

\section{Preliminaries}

Let $m(x)$ be a non-decreasing and right continuous function on $\left(
-\infty,\infty\right)  $ satisfying%
\[
m(-\infty)=0,\text{ \ \ \ }m(\infty)\leq\infty.
\]
Set%
\[
l=\sup\left\{  x>-\infty,\text{ }m(x)<+\infty\right\}  ,\text{ \ }l_{+}%
=\sup\text{\textrm{supp}}dm,\text{ \ }l_{-}=\inf\text{\textrm{supp}}dm.
\]
Note $m(l)=\infty$ if $l<\infty.$ Assume%

\begin{equation}%
{\displaystyle\int\nolimits_{-\infty}^{a}}
\left\vert x\right\vert dm(x)<\infty\label{1}%
\end{equation}
with some $a\in\left(  l_{-},l_{+}\right)  $. Let $\mathcal{E}$ be the
totality of non-decreasing functions $m$ satisfying (\ref{1}). We exclude $m$
vanishing identically on $\left(  -\infty,\infty\right)  $ from $\mathcal{E}$.
One can regard $dm$ as a distribution of weight and in this case $m$ works as
a string. On the other hand, one can associate a generalized diffusion process
with generator $L$%
\[
L=\dfrac{d}{dm}\dfrac{d}{dx}%
\]
if we impose a suitable boundary condition if necessary. The condition
(\ref{1}) is called as entrance condition in 1D diffusion theory developed by
W.Feller, so we call $m$ satisfying (\ref{1}) a string of entrance type. For
an entrance type $m$, it is easy to show that for $\lambda$ $\in\mathbf{C}%
$\ an integral equation%
\[
\varphi(x)=1-\lambda%
{\displaystyle\int\nolimits_{-\infty}^{x}}
(x-y)\varphi(y)dm(y)
\]
has a unique solution, which is denoted by $\varphi_{\lambda}(x)$. Introduce a subspace%

\[
L_{0}^{2}(dm)=\left\{  f\in L^{2}(dm);\text{ supp}f\subset(-\infty,l)\right\}
,
\]
and for $f\in L_{0}^{2}(dm)$ define a generalized Fourier transform by%
\[
\widehat{f}(\lambda)=%
{\displaystyle\int\nolimits_{-\infty}^{l}}
f(x)\varphi_{\lambda}(x)dm(x).
\]
Krein's spectral theory implies there exists a measure $\sigma$ on
$[0,\infty)$ satisfying%
\begin{equation}%
{\displaystyle\int\nolimits_{-\infty}^{l}}
\left\vert f(x)\right\vert ^{2}dm(x)=%
{\displaystyle\int\nolimits_{0}^{\infty}}
\left\vert \widehat{f}(\xi)\right\vert ^{2}d\sigma(\xi)\text{ \ \ for any
}f\in L_{0}^{2}(dm).\label{2}%
\end{equation}
$\sigma$ is called a spectral measure for the string $m.$ The non-uniqueness
of such $\sigma$ occurs if and only if%
\begin{equation}
l_{+}+m\left(  l_{+}\right)  <\infty.\label{3}%
\end{equation}
The number $l\left(  \geq l_{+}\right)  $ possesses its meaning only when
(\ref{3}) is satisfied, and in this case there exists a $\sigma$ satisfying
(\ref{2}) with the boundary condition%
\[
f\left(  l_{+}\right)  +\left(  l-l_{+}\right)  f^{+}\left(  l_{+}\right)  =0
\]
at $l_{+}$. Here $f^{+\text{ }}$is the derivative from the right hand side. If
$l=\infty,$ this should be interpreted as%
\[
f^{+}\left(  l_{+}\right)  =0.
\]
At the left boundary $l_{-}$ no boundary condition is necessary if
$l_{-}=-\infty,$ and if $l_{-}>-\infty$ we impose the reflective boundary
condition, namely%
\[
f^{-}\left(  l_{-}\right)  =0\text{ \ \ the derivative from left.}%
\]
Generally, for a string $m$ of entrance type it is known that for $\lambda<0$
there exists uniquely $f$ such that%
\[
\left\{
\begin{array}
[c]{l}%
-Lf=\lambda f,\text{ \ }f>0,\text{ \ }f^{+}\leq0,\text{ \ }f\left(  l-\right)
=0\smallskip\\
f(x)\varphi_{\lambda}^{+}(x)-f^{+}(x)\varphi_{\lambda}(x)=1
\end{array}
.\right.
\]
This unique $f$ is denoted by $f_{\lambda}$ and contains information of the
boundary condition we are imposing on $-L$ at the right boundary $l_{+}$, and
$f_{\lambda}$ can be represented by $\varphi_{\lambda}$ as%
\begin{equation}
f_{\lambda}(x)=\varphi_{\lambda}(x)%
{\displaystyle\int\nolimits_{x}^{l}}
\frac{dy}{\varphi_{\lambda}(y)^{2}}.\label{4}%
\end{equation}
The right side integral is always convergent for $\lambda<0$, because if
\textrm{supp}$dm\neq\phi,$ then choosing $a\in\mathrm{supp}dm,$ we see for
$x>a$%
\[
\varphi_{\lambda}(x)\geq1-\lambda\int\nolimits_{-\infty}^{x}\left(
x-y\right)  dm(y)\geq1-\lambda\int\nolimits_{-\infty}^{a}\left(  x-y\right)
dm(y)\geq1-\lambda\left(  x-a\right)  m(a),
\]
hence%
\[%
{\displaystyle\int\nolimits_{x}^{l}}
\frac{dy}{\varphi_{\lambda}(y)^{2}}\leq%
{\displaystyle\int\nolimits_{x}^{l}}
\frac{dy}{\left(  1-\lambda\left(  y-a\right)  m(a)\right)  ^{2}}<\infty
\]
for $x>a.$ If \textrm{supp}$dm=\phi,$ then $l<\infty$ and%
\begin{equation}
m(x)=\left\{
\begin{array}
[c]{ccc}%
0 & \text{for} & \text{ }x<l\\
\infty & \text{for} & x>l
\end{array}
,\right.  \label{13}%
\end{equation}
which implies%
\[
\varphi_{\lambda}(x)=\left\{
\begin{array}
[c]{ccc}%
1 & \text{for} & x<l\\
\infty & \text{for} & x>l
\end{array}
,\right.
\]
and%
\[%
{\displaystyle\int\nolimits_{x}^{l}}
\frac{dy}{\varphi_{\lambda}(y)^{2}}=l-x<\infty.
\]
Here note that we have excluded $m=0$ identically on $\left(  -\infty
,\infty\right)  ,$ hence $l<\infty.$ If $m$ is a non-decreasing function of
(\ref{13}) the spectral measure vanishes identically on $[0,\infty).$ If $m$
is $\infty$ identically on $\left(  -\infty,\infty\right)  ,$ the spectral
function $\sigma$ is defined to be $0$ identically on $[0,\infty).$ Conversely
if a spectral measure vanishes identically on $[0,\infty),$ then the
associated string $m$ should be of (\ref{13}). $\varphi_{\lambda}(x)$ is an
entire function of minimal exponential type as a function of $\lambda$ and the
zeroes of $\varphi_{\lambda}(x)$ coincide with the eigenvalues of $-L$ defined
as a self-adjoint operator on $L^{2}(dm,(-\infty,x])$ with the Dirichlet
boundary condition at $x,$ which means that $\varphi_{\lambda}(x)$ has simple
zeroes on $(0,\infty).$ The Green function $g_{\lambda}$ for $-L$ on
$L^{2}(dm)$ is given by%
\[
g_{\lambda}(x,y)=g_{\lambda}(y,x)=f_{\lambda}(y)\varphi_{\lambda}(x)
\]
for $x\leq y.$ The relationship between $\sigma$ and $g_{\lambda}$ is
described by an identity%
\[%
{\displaystyle\int\nolimits_{-\infty}^{l}}
{\displaystyle\int\nolimits_{-\infty}^{l}}
g_{\lambda}(x,y)f(x)\overline{f(y)}dm(x)dm(y)=%
{\displaystyle\int\nolimits_{0}^{\infty}}
\frac{\left\vert \widehat{f}(\xi)\right\vert ^{2}}{\xi-\lambda}\sigma(d\xi)
\]
for any $f\in L^{2}(dm),$ and%
\[
g_{\lambda}(x,y)=%
{\displaystyle\int\nolimits_{0}^{\infty}}
\frac{\varphi_{\xi}(x)\varphi_{\xi}(y)}{\xi-\lambda}d\sigma\left(  \xi\right)
,
\]
through which $\sigma$ is determined uniquely from the string $m.$ Distinct
$m$s may give an equal $\sigma,$ namely for $a\in\mathbf{R}$ a new string%
\[
m_{a}\left(  x\right)  =m\left(  x+a\right)
\]
defines the same $\sigma,$ because%
\[
\varphi_{\lambda}^{a}(x)=\varphi_{\lambda}(x+a),\text{ \ }f_{\lambda}%
^{a}(x)=f_{\lambda}(x+a),
\]
hence%
\[
g_{\lambda}^{a}(x,x)=\varphi_{\lambda}^{a}(x)f_{\lambda}^{a}(x)=g_{\lambda
}(x+a,x+a)=%
{\displaystyle\int\nolimits_{0}^{\infty}}
\frac{\varphi_{\xi}(x+a)^{2}}{\xi-\lambda}d\sigma\left(  \xi\right)  .
\]
On the other hand%
\[
g_{\lambda}^{a}(x,x)=%
{\displaystyle\int\nolimits_{0}^{\infty}}
\frac{\varphi_{\xi}^{a}(x)\varphi_{\xi}^{a}(x)}{\xi-\lambda}d\sigma_{a}\left(
\xi\right)  =%
{\displaystyle\int\nolimits_{0}^{\infty}}
\frac{\varphi_{\xi}(x+a)^{2}}{\xi-\lambda}d\sigma_{a}\left(  \xi\right)  ,
\]
hence an identity%
\[
\sigma_{a}\left(  \xi\right)  =\sigma\left(  \xi\right)
\]
should be held. Conversely we have

\begin{theorem}
\label{t1}$($Kotani\cite{k1,k2}$)$ If two strings $m_{1}$ and $m_{2}$ of
$\mathcal{E}$ have the same spectral measure $\sigma,$ then $m_{1}(x+c)=$
$m_{2}(x)$ for a $c\in\mathbf{R}.$
\end{theorem}

If we hope to obtain the continuity of the correspondence between $m$ and
$\sigma,$ we have to keep the non-uniqueness in mind. Namely, for $m$ of
$\mathcal{E}$ a sequence $\left\{  m_{n}\right\}  _{n\geq1}$ of $\mathcal{E}$
defined by%
\[
m_{n}\left(  x\right)  =m\left(  x-n\right)
\]
converges to the trivial function $0$ as $n\rightarrow\infty.$ However the
associated $\sigma$s are independent of $n.$ Therefore we shall give several
alternative definitions of convergence by imposing certain extra conditions
(related to tightness) in addition to pointwise convergence. Set%
\[
M(x)=%
{\displaystyle\int\nolimits_{-\infty}^{x}}
\left(  x-y\right)  dm(y)=%
{\displaystyle\int\nolimits_{-\infty}^{x}}
m(y)dy.
\]
Then, the condition (\ref{1}) is equivalent to%
\[
M(x)<\infty
\]
for $x<l.$ Using a convention%
\[
\lbrack-\infty,a)=\left(  -\infty,a\right)  ,\text{ }(a,\infty]=\left(
a,\infty\right)  \text{ and so on,}%
\]
we can see $M$ is a non-decreasing convex function on $\left(  -\infty
,\infty\right)  $ satisfying%
\[
\left\{
\begin{array}
[c]{l}%
M\left(  x\right)  =0\text{ \ on }(-\infty,l_{-}],\\
\text{continuous and strictly increasing on }[l_{-},l),\\
M\left(  x\right)  =\infty\text{ on }\left(  l,\infty\right)  .
\end{array}
\right.
\]
For a fixed positive number $c$, we assume%
\begin{equation}
0\in(l_{-},l]\text{ and }M\left(  l\right)  \geq c,\label{5}%
\end{equation}
and normalize such an $m$ by%
\begin{equation}
M\left(  0\right)  =c.\label{9}%
\end{equation}
Denote by $\mathcal{E}^{\left(  c\right)  }$ the set of all elements of
$\mathcal{E}$ satisfying (\ref{5}), (\ref{9}) and set%
\[
\mathcal{E}_{+}=\bigcup\limits_{c>0}\mathcal{E}^{\left(  c\right)  }.
\]
In this definition of $\mathcal{E}_{+}$ among functions satisfying (\ref{1})
any function $m$ defined by (\ref{13}) for some $l\leq\infty$ is excluded from
$\mathcal{E}_{+}.$ Therefore, $\mathcal{E}\backslash\mathcal{E}_{+}$ consists
of $m$ satisfying (\ref{13}) for some $l<\infty.$ The uniqueness of the
correspondence between $m$ and $\sigma$ holds under this normalization. Set%
\[
\mathcal{S}=\text{the set of all spectral measures for strings of }%
\mathcal{E}.
\]
Any suitable characterization of $\mathcal{S}$ is not known yet, however, any
measure on $[0,\infty)$ with polynomial growth at $\infty$ belongs to
$\mathcal{S}$.

We prepare a basic estimate for $\varphi_{\lambda}.$ $\varphi_{\lambda}$ can
be represented as%
\begin{equation}
\varphi_{\lambda}\left(  x\right)  =%
{\displaystyle\sum\limits_{n=0}^{\infty}}
\left(  -\lambda\right)  ^{n}\phi_{n}(x),\label{6}%
\end{equation}
where $\left\{  \phi_{n}\right\}  _{n\geq0}$ are%
\[
\phi_{n}(x)=%
{\displaystyle\int\nolimits_{-\infty}^{x}}
\left(  x-y\right)  \phi_{n-1}(y)dm(y),\text{ \ \ \ }\phi_{0}(x)=1.
\]
Then, the convergence of the above series can be seen by

\begin{lemma}
\label{l1}$\varphi_{\lambda}$ is given by an absolute convergent series
$($\ref{6}$)$ and satisfies%
\[
\left\vert \varphi_{\lambda}\left(  x\right)  \right\vert \leq\exp\left(
\left\vert \lambda\right\vert M(x)\right)  .
\]

\end{lemma}

\begin{proof}
First we show for any $k\geq0$%
\begin{equation}
\phi_{k}(x)\leq\frac{M(x)^{k}}{k!}\label{7}%
\end{equation}
holds. Observe%
\[
\phi_{1}(x)=%
{\displaystyle\int\nolimits_{-\infty}^{x}}
\left(  x-y\right)  dm(y)=M(x).
\]
Assuming (\ref{7}) for some $k$, we have%
\begin{align*}
\phi_{k+1}(x) &  \leq\frac{1}{k!}%
{\displaystyle\int\nolimits_{-\infty}^{x}}
\left(  x-y\right)  M(y)^{k}dm(y)\\
&  =\frac{1}{k!}%
{\displaystyle\int\nolimits_{-\infty}^{x}}
\left(  M(y)-k\left(  x-y\right)  M^{\prime}(y)\right)  M(y)^{k-1}M^{\prime
}(y)dy\\
&  \leq\frac{1}{k!}%
{\displaystyle\int\nolimits_{-\infty}^{x}}
M^{\prime}(y)M(y)^{k}dy=\frac{M(x)^{k+1}}{\left(  k+1\right)  !},
\end{align*}
which proves (\ref{7}) for general $k.$ Then the estimate of $\varphi
_{\lambda}$ is clear.
\end{proof}

Here we clarify the convergence of a sequence of monotone functions taking
value $\infty.$ For non-negative and non-decreasing function $m$ which may
take $\infty,$ set%
\[
\widehat{m}\left(  x\right)  =\frac{2}{\pi}\tan^{-1}m(x),\text{ \ \ \ }%
x\in\boldsymbol{R}.
\]
Then%
\[
\widehat{m}\left(  x\right)  \in\left[  0,1\right]
\]
and right continuous non-decreasing function satisfying%
\[
0\leq\widehat{m}\left(  -\infty\right)  \leq\widehat{m}\left(  x\right)
\leq\widehat{m}\left(  l-\right)  \leq\widehat{m}\left(  l\right)  =1,
\]
if $l<\infty.$ A sequence of non-negative and non-decreasing functions $m_{n}$
is defined to converge to $m$ as $n\rightarrow\infty$ if%
\begin{equation}
\widehat{m}_{n}\left(  x\right)  \rightarrow\widehat{m}\left(  x\right)
\end{equation}
holds at any point of continuity of $\widehat{m}\left(  x\right)  $.

\begin{lemma}
\label{l3}Suppose $m_{n}\in\mathcal{E}$ converges to $m\in\mathcal{E}$ as
$n\rightarrow\infty.$ Then it holds that%
\[
\underset{n\rightarrow\infty}{\underline{\lim}}l_{n}\geq l.
\]

\end{lemma}

\begin{proof}
Let $x<l$ be a point of continuity for $\widehat{m}.$ Then%
\[
\widehat{m}_{n}\left(  x\right)  \rightarrow\widehat{m}\left(  x\right)  <1,
\]
hence%
\[
\widehat{m}_{n}\left(  x\right)  <1
\]
for every sufficiently large $n,$ which implies $x<l_{n}$ and completes the proof.
\end{proof}

The continuity of the correspondence from $\mathcal{E}$ to $\mathcal{S}$ is
not hard to show. Let $m_{n},m$\ be strings of $\mathcal{E}$ and define the
convergence of $m_{n}$ to $m$ by\medskip\newline(A) $m_{n}\left(  x\right)
\rightarrow m\left(  x\right)  $ for every point of continuity of
$m.\smallskip$\newline(B) $\lim\limits_{x\rightarrow-\infty}\sup
\limits_{n\geq1}M_{n}(x)=0$\newline

\begin{theorem}
\label{t9}Suppose $m_{n}\in\mathcal{E}$ converge to $m\in\mathcal{E}$. Then,
for every $\lambda<0$ the Green functions $g_{\lambda}^{\left(  n\right)
}\left(  x,y\right)  $ of the string $m_{n}$ converge to the Green function
$g_{\lambda}\left(  x,y\right)  $ of $m$ for any $x,y<l.$ In particular the
spectral functions $\sigma_{n}\left(  \xi\right)  $ converge to $\sigma\left(
\xi\right)  $ at every point of continuity of $\sigma.$
\end{theorem}

\begin{proof}
Under the conditions it is easy to see that the $\varphi-$functions
$\varphi_{\lambda}^{\left(  n\right)  }(x)$ of $m_{n}$ converge to the
$\varphi-$function $\varphi_{\lambda}(x)$ of $m$ compact uniformly with
respect to $\left(  x,\lambda\right)  \in\left(  -\infty,l\right)
\times\boldsymbol{C}$ from the uniform bound for $\varphi_{\lambda}^{\left(
n\right)  }$ due to Lemma \ref{l1}.$.$ Moreover, if $m\left(  a\right)  >0$ at
some $a,$ a point of continuity of $m$, then there exists a positive constant
$C$ such that%
\[
\varphi_{\lambda}^{\left(  n\right)  }(y)\geq1-\lambda M_{n}\left(  y\right)
\geq1+C\left(  y-a\right)
\]
holds for any $y>a$, hence%
\[
f_{\lambda}^{\left(  n\right)  }(x)=\varphi_{\lambda}^{\left(  n\right)
}(x)\int_{x}^{l_{n}}\frac{1}{\varphi_{\lambda}^{\left(  n\right)  }(y)^{2}}dy
\]
also converge to $f_{\lambda}(x).$ If supp$m=\phi,$ namely $m(x)=0$
identically on $\left(  -\infty,l\right)  ,$ then%
\[
f_{\lambda}^{\left(  n\right)  }(x)\rightarrow l-x
\]
if $l<\infty.$ The case $l=\infty$ is excluded. Consequently, we have%
\[
g_{\lambda}^{\left(  n\right)  }\left(  x,y\right)  =\varphi_{\lambda
}^{\left(  n\right)  }(y)f_{\lambda}^{\left(  n\right)  }(x)\rightarrow
\varphi_{\lambda}(y)f_{\lambda}(x)=g_{\lambda}\left(  x,y\right)
\]
for any $y\leq x<l.$ The identity%
\[
g_{\lambda}^{\left(  n\right)  }\left(  x,y\right)  =\int_{0}^{\infty}%
\frac{\varphi_{\xi}^{\left(  n\right)  }(x)\varphi_{\xi}^{\left(  n\right)
}(y)}{\xi-\lambda}d\sigma_{n}\left(  \xi\right)
\]
shows the last statement of the theorem.
\end{proof}

\section{Scales and estimates by trace}

The straight converse statement of the theorem \ref{t9} is hopeless to be
true, because there is no characterization for a measure $\sigma$ on
$[0,\infty)$ to be a spectral measure of a string $m\in\mathcal{E}$ $.$
Therefore we prove the converse continuity of the correspondence by imposing a
condition on $\left\{  \sigma_{n}\right\}  .$ In the process of the proof we
have to estimate $\varphi_{\lambda}(x)^{-2}$ in terms of $m.$ A better way to
investigate $\varphi_{\lambda}(x)^{-2}$ is to use probabilistic methods.
Recall that for each fixed $a,$ $\varphi_{\lambda}(a)$ has simple zeroes
$\left\{  \mu_{n}\right\}  _{n\geq1}$ which are eigenvalues of $-L$ on
$(-\infty,a]$ with Dirichlet boundary condition at $x=a.$ Since the Green
function for this operator is%
\[
a-\left(  x\vee y\right)  ,
\]
we see%
\begin{equation}
\sum\limits_{n=1}^{\infty}\mu_{n}^{-1}=\mathrm{tr}\left(  -L\right)  ^{-1}=%
{\displaystyle\int\nolimits_{-\infty}^{a}}
\left(  a-x\right)  dm(x)=M(a)<\infty.\label{35}%
\end{equation}
Choosing a $b<l$, we denote by $\phi_{\lambda}^{b}$ $\left(  \psi_{\lambda
}^{b}\right)  $ the solutions of%
\[
-\dfrac{d}{dm}\dfrac{d}{dx}f=\lambda f,\text{ \ with }f(b)=1,\text{ }%
f^{\prime}(b)=0\text{ \ }\left(  f(b)=0,\text{ }f^{\prime}(b)=1\right)  \text{
respectively.}%
\]
Then we see an identity%
\[
\varphi_{\lambda}(x)=\varphi_{\lambda}(b)\phi_{\lambda}^{b}\left(  x\right)
+\varphi_{\lambda}^{\prime}(b)\psi_{\lambda}^{b}\left(  x\right)
\]
holds. Lemma \ref{l1} implies $\varphi_{\lambda}(b)$ and $\varphi_{\lambda
}^{\prime}(b)$ are entire functions of at most exponential type $M(b)$ as
functions of $\lambda$. On the other hand, $\phi_{\lambda}^{b}\left(
x\right)  $ and $\psi_{\lambda}^{b}\left(  x\right)  $ are entire functions of
order at most $1/2$ as functions of $\lambda$. Therefore we know that
$\varphi_{\lambda}(x)$ is an entire function of minimal exponential type,
which combined with (\ref{35}) shows%
\[
\varphi_{\lambda}(a)=\prod\limits_{n=1}^{\infty}\left(  1-\frac{\lambda}%
{\mu_{n}}\right)  .
\]
For the detail refer to page 441 of \cite{k1}. Now let $\left\{
X_{n}\right\}  _{n\geq1}$ be independent random variables each of which has an
exponential distribution of mean $1.$ Then, an identity%
\[
E\exp\left(  \lambda\sum_{n=1}^{\infty}\mu_{n}^{-1}X_{n}\right)
=\prod\limits_{n=1}^{\infty}\left(  1-\frac{\lambda}{\mu_{n}}\right)
^{-1}=\varphi_{\lambda}(a)^{-1}%
\]
holds. Therefore, letting $\left\{  \widetilde{X}_{n}\right\}  _{n\geq1}$ be
independent copies of $\left\{  X_{n}\right\}  _{n\geq1}$ and setting%
\[
Y_{n}=X_{n}+\widetilde{X}_{n},\text{ \ }X=\sum_{n=1}^{\infty}\mu_{n}^{-1}%
Y_{n},
\]
we have%
\begin{equation}
\varphi_{\lambda}(a)^{-2}=E\exp\left(  \lambda X\right)  .\label{10}%
\end{equation}
We denote $X=X(a)$ if necessary, because the eigenvalues $\left\{  \mu
_{n}\right\}  _{n\geq1}$ depends on the boundary $a.$

\begin{lemma}
\label{l15}Suppose the spectral measure $\sigma$ of an $m\in\mathcal{E}$
satisfies%
\[
p(t)=\int\nolimits_{0}^{\infty}e^{-t\xi}d\sigma\left(  \xi\right)
<\infty\text{ \ for any }t>0.
\]
Then, for any non-negative Borel measurable function $f$ on $[0,\infty)$%
\begin{equation}
\int\nolimits_{-\infty}^{l}Ef\left(  X\left(  x\right)  \right)
dx=\int\nolimits_{0}^{\infty}p(t)f(t)dt \label{26}%
\end{equation}
holds by permitting for the integrals to take the value $\infty$ simultaneously.
\end{lemma}

\begin{proof}
From (\ref{4}) it follows that for any $x<l$ and $\lambda<0$%
\[
\int\nolimits_{x}^{l}Ee^{\lambda X\left(  y\right)  }dy=Ee^{\lambda X\left(
x\right)  }\int\nolimits_{0}^{\infty}\frac{\varphi_{\xi}\left(  x\right)
^{2}}{\xi-\lambda}d\sigma\left(  \xi\right)  =\int\nolimits_{0}^{\infty
}Ee^{\lambda\left(  X\left(  x\right)  +t\right)  }p(t,x,x)dt,
\]
holds, where $p(t,x,y)$ is the transition probability density defined by%
\[
p(t,x,y)=\int\nolimits_{0}^{\infty}e^{-t\xi}\varphi_{\xi}\left(  x\right)
\varphi_{\xi}\left(  y\right)  d\sigma\left(  \xi\right)  ,
\]
hence a functional monotone class theorem shows that the identity below holds
for any non-negative bounded continuous function $f$ on $[0,\infty).$%
\begin{equation}
\int\nolimits_{x}^{l}E\left(  f\left(  X\left(  y\right)  \right)  e^{\lambda
X\left(  y\right)  }\right)  dy=\int\nolimits_{0}^{\infty}E\left(  f\left(
X\left(  x\right)  +t\right)  e^{\lambda\left(  X\left(  x\right)  +t\right)
}\right)  p(t,x,x)dt \label{27}%
\end{equation}
Since, for $t>2M(x)$%
\[
p(t,x,x)=\int\nolimits_{0}^{\infty}e^{-\xi t}\varphi_{\xi}\left(  x\right)
^{2}d\sigma\left(  \xi\right)  \leq\int\nolimits_{0}^{\infty}e^{-\xi t}e^{2\xi
M(x)}d\sigma\left(  \xi\right)  =p\left(  t-2M(x)\right)
\]
holds, assuming $f(t)=0$ for $t<\epsilon$, we see%
\[
\int\nolimits_{-\infty}^{l}E\left(  f\left(  X\left(  y\right)  \right)
e^{\lambda X\left(  y\right)  }\right)  dy=\int\nolimits_{0}^{\infty}f\left(
t\right)  e^{\lambda t}p(t)dt
\]
by letting $x\rightarrow-\infty.$ Here we have used the fact%
\[
X\left(  x\right)  \rightarrow0\text{ \ \ \ as }x\rightarrow-\infty.
\]
The rest of the proof is a routine.
\end{proof}

Now we define a scale function $\phi$ on $\left[  0,1\right]  $ as a function
satisfying the following properties.\medskip\newline(S.1) \ $\phi$ is strictly
increasing, convex and $\phi\left(  0\right)  =0,$ $\phi^{\prime}\left(
1-\right)  <\infty.$\newline(S.2) \ For each $x>0$%
\[
\overline{\lim_{y\downarrow0}}\frac{\phi\left(  xy\right)  }{\phi\left(
y\right)  }<\infty.
\]
(S.3) \ For each $x\in(0,1]$%
\[
\underset{y\downarrow0}{\underline{\lim}}\frac{\phi\left(  xy\right)  }%
{\phi\left(  y\right)  }>0
\]
The property (S.1) enables us to extend $\phi$ linearly to $[1,\infty)$,
namely%
\[
\phi\left(  x\right)  =\phi\left(  1\right)  +\phi^{\prime}\left(  1-\right)
\left(  x-1\right)
\]
for $x>1.$ Then $\phi$ becomes non-negative, convex and non-decreasing
function on $[0,\infty).$ Throughout the paper $\phi$ is always extended to
$[1,\infty)$ linearly in this way. A regularly varying function at $0$
satisfies the condition (S.2), (S.3). Set%
\[
C_{+}(x)=\sup_{y>0}\frac{\phi\left(  xy\right)  }{\phi\left(  y\right)
}<\infty,\text{ \ \ \ }C_{-}(x)=\inf_{y\in(0,1]}\frac{\phi\left(  xy\right)
}{\phi\left(  y\right)  }>0.
\]
Then $C_{+}$ becomes non-negative, convex and non-decreasing on $[0,\infty).$
It satisfies the submultiplicative property%
\[
C_{+}(xy)\leq C_{+}(x)C_{+}(y)
\]
for any $x,y>0,$ hence, setting%
\[
\alpha_{+}=\sup_{x>1}\frac{\log C_{+}\left(  e^{x}\right)  }{x}\in
\lbrack0,\infty)
\]
we see%
\[
C_{+}(x)\leq x^{\alpha_{+}}%
\]
holds for any $x\geq e.$ $\alpha_{+}$ should be not less than $1$ due to the
convexity of $C_{+}.$ Since%
\[
\phi\left(  xy\right)  \leq C_{+}(x)\phi\left(  y\right)
\]
holds for any $x,y>0,$ we have%
\[
\phi\left(  x\right)  \geq\frac{\phi\left(  1\right)  }{C_{+}\left(
1/x\right)  }\geq\phi\left(  1\right)  x^{\alpha_{+}}%
\]
for any $x\in\left[  0,1/e\right]  .$ Therefore the property (S.2) restricts
$\phi$ not to decay faster than with a power order.%
\begin{equation}
\phi\left(  xy\right)  \leq C_{+}(x)\phi\left(  y\right)  \label{22}%
\end{equation}
$C_{-}$ satisfies%
\[
C_{-}(xy)\geq C_{-}(x)C_{-}(y)
\]
for any $x,y\in\left[  0,1\right]  .$ Typical examples for functions
satisfying (S.1)$\thicksim$(S.3) are%
\[
\phi\left(  x\right)  =x^{\alpha},\text{ \ \ }x^{\alpha}\left(  c-\log
x\right)  ,
\]
where $\alpha\geq1$ and $c$ is a sufficiently large positive constant.

\begin{lemma}
\label{l9}Let $\left\{  Y_{n}\right\}  _{n\geq1}$ be a sequence of identically
distributed non-negative random variables with mean $\mu.$ and $\left\{
\lambda_{n}\right\}  _{n\geq1}$ be a non-negative sequence satisfying%
\[
\sum_{n=1}^{\infty}\lambda_{n}<\infty,
\]
and set%
\[
X=\sum_{n=1}^{\infty}\lambda_{n}Y_{n}.
\]
Then, we have\newline$(1)$ \ If $\phi$ satisfies $(S.1),$ then%
\[
\phi\left(  EX\right)  \leq E\phi\left(  X\right)  .
\]
$(2)$ \ If $\phi$ satisfies $(S.1),(S.2),$ then%
\[
E\phi\left(  X\right)  \leq\left(  EC_{+}\left(  \frac{Y_{1}}{\mu}\right)
\right)  \phi\left(  EX\right)  .
\]

\end{lemma}

\begin{proof}
Jensen's inequality implies the inequality in (1). To show the second
inequality we set%
\[
m_{n}=m^{-1}\lambda_{n},\text{ \ \ }m=\sum_{k=1}^{\infty}\lambda_{k}.
\]
Then (\ref{22}) implies%
\[
\phi\left(  X\right)  =\phi\left(  m\mu\sum_{n=1}^{\infty}m_{n}\frac{Y_{n}%
}{\mu}\right)  \leq\phi\left(  m\mu\right)  C_{+}\left(  \sum_{n=1}^{\infty
}m_{n}\frac{Y_{n}}{\mu}\right)
\]
Since the function $C_{+}$ is convex, we have%
\[
C_{+}\left(  \sum_{n=1}^{\infty}m_{n}\frac{Y_{n}}{\mu}\right)  \leq\sum
_{n=1}^{\infty}m_{n}C_{+}\left(  \frac{Y_{n}}{\mu}\right)  ,
\]
and%
\[
E\phi\left(  X\right)  \leq\phi\left(  m\mu\right)  \sum_{n=1}^{\infty}%
m_{n}EC_{+}\left(  \frac{Y_{n}}{\mu}\right)  =\phi\left(  EX\right)
EC_{+}\left(  \frac{Y_{1}}{\mu}\right)  .
\]

\end{proof}

Let $X$ be the same as was defined in (\ref{10}) and set%
\[
C_{\phi}=EC_{+}\left(  \frac{Y_{1}}{\mu}\right)  =\int_{0}^{\infty}%
te^{-t}C_{+}\left(  t/2\right)  dt<\infty.
\]

\begin{lemma}
\label{l7}We have\newline$(1)$ If $\phi$ satisfies $(S.1),$ then%
\[
E\left(  \phi\left(  X\right)  e^{\lambda X}\right)  \geq\varphi_{\lambda
}\left(  a\right)  ^{-2}\phi\left(  \int\nolimits_{-\infty}^{a}\varphi
_{\lambda}\left(  x\right)  ^{2}dm\left(  x\right)  \int\nolimits_{x}%
^{a}\varphi_{\lambda}\left(  y\right)  ^{-2}dy\right)  .
\]
$(2)$\ If $\phi$ satisfies $(S.1),(S.2),$ then%
\[
E\left(  \phi\left(  X\right)  e^{\lambda X}\right)  \leq C_{\phi}%
\varphi_{\lambda}\left(  a\right)  ^{-2}\phi\left(  \int\nolimits_{-\infty
}^{a}\varphi_{\lambda}\left(  x\right)  ^{2}dm\left(  x\right)  \int
\nolimits_{x}^{a}\varphi_{\lambda}\left(  y\right)  ^{-2}dy\right)  .
\]

\end{lemma}

\begin{proof}
For a fixed $\lambda<0$ let $Z$ be a non-negative random variable satisfying%
\[
Ee^{\mu Z}=\varphi_{\lambda+\mu}\left(  a\right)  ^{-2}\varphi_{\lambda
}\left(  a\right)  ^{2}=\prod\limits_{n=1}^{\infty}\left(  1-\frac{\mu}%
{\mu_{n}-\lambda}\right)  ^{-2}.
\]
Then, note an identity%
\begin{equation}
E\left(  \phi\left(  X\right)  e^{\lambda X}\right)  =\varphi_{\lambda}\left(
a\right)  ^{-2}E\left(  \phi\left(  Z\right)  \right)  , \label{32}%
\end{equation}
which can be shown from the observation%
\[
\frac{\partial^{k}}{\partial\lambda^{k}}\varphi_{\lambda}\left(  a\right)
^{-2}=\varphi_{\lambda}\left(  a\right)  ^{-2}\left.  \frac{\partial^{k}%
}{\partial\mu^{k}}\left(  \varphi_{\lambda+\mu}\left(  a\right)  ^{-2}%
\varphi_{\lambda}\left(  a\right)  ^{2}\right)  \right\vert _{\mu=0}%
\]
for any $k\geq0,$ because this implies the identity when $\phi\left(
x\right)  =x^{k}.$ To apply Lemma \ref{l9} to $Z$ we need to compute $EZ.$ If
we denote the Green operator for $L$ on $(-\infty,a]$ with Dirichlet boundary
condition at $a$ by $G_{\lambda},$ then%
\[
G_{\lambda}\left(  x,y\right)  =\varphi_{\lambda}\left(  y\right)
\varphi_{\lambda}\left(  x\right)  \int_{x}^{a}\varphi_{\lambda}\left(
z\right)  ^{-2}dz\text{ \ \ for }x\geq y
\]
hence%
\[
EZ=\sum_{j=1}^{\infty}\frac{1}{\mu_{j}-\lambda}=\mathrm{tr}G_{\lambda}%
=\int\nolimits_{-\infty}^{a}\varphi_{\lambda}\left(  x\right)  ^{2}dm\left(
x\right)  \int\nolimits_{x}^{a}\varphi_{\lambda}\left(  y\right)  ^{-2}dy
\]
holds, and we have the inequalities in the statement.
\end{proof}

The right side of the inequalities in Lemma \ref{l7} can be estimated further.

\begin{lemma}
\label{l2}For $\lambda<0$ the following inequalities are valid.\newline$(1)$ $%
{\displaystyle\int\nolimits_{-\infty}^{a}}
\varphi_{\lambda}\left(  x\right)  ^{2}dm\left(  x\right)
{\displaystyle\int\nolimits_{x}^{a}}
\varphi_{\lambda}\left(  y\right)  ^{-2}dy\geq M(a)\varphi_{\lambda}\left(
a\right)  ^{-2}\medskip$\newline$(2)$ $%
{\displaystyle\int\nolimits_{-\infty}^{a}}
\varphi_{\lambda}\left(  x\right)  ^{2}dm\left(  x\right)
{\displaystyle\int\nolimits_{x}^{a}}
\varphi_{\lambda}\left(  y\right)  ^{-2}dy\leq M(a)\wedge\left(  \dfrac
{\log\varphi_{\lambda}\left(  a\right)  }{-\lambda}\right)  $
\end{lemma}

\begin{proof}
The inequality (1) and the first inequality of (2) follow from the
monotonicity of $\varphi_{\lambda}\left(  z\right)  $, namely we have%
\[
\int\nolimits_{x}^{a}\varphi_{\lambda}\left(  y\right)  ^{-2}dy\geq
\varphi_{\lambda}\left(  a\right)  ^{-2}\left(  a-x\right)  ,\text{ \ }%
\int\nolimits_{x}^{a}\varphi_{\lambda}\left(  y\right)  ^{-2}dy\leq
\varphi_{\lambda}\left(  x\right)  ^{-2}\left(  a-x\right)  ,
\]
which implies%
\[
\int\nolimits_{-\infty}^{a}\varphi_{\lambda}\left(  x\right)  ^{2}dm\left(
x\right)  \int\nolimits_{x}^{a}\varphi_{\lambda}\left(  y\right)  ^{-2}%
dy\geq\varphi_{\lambda}\left(  a\right)  ^{-2}\int\nolimits_{-\infty}%
^{a}\left(  a-x\right)  dm\left(  x\right)  =\varphi_{\lambda}\left(
a\right)  ^{-2}M(a),
\]
and%
\[
\int\nolimits_{-\infty}^{a}\varphi_{\lambda}\left(  x\right)  ^{2}dm\left(
x\right)  \int\nolimits_{x}^{a}\varphi_{\lambda}\left(  y\right)  ^{-2}%
dy\leq\int\nolimits_{-\infty}^{a}\left(  a-x\right)  dm(x)=M(a).
\]
The second inequality of (2) follows by using the equation satisfied by
$\varphi_{\lambda}\left(  x\right)  $%
\[
d\varphi_{\lambda}^{\prime}\left(  y\right)  =-\lambda\varphi_{\lambda}\left(
y\right)  dm\left(  y\right)  ,
\]
which yields%
\begin{align*}
&  -\lambda\int\nolimits_{-\infty}^{a}\varphi_{\lambda}\left(  y\right)
^{2}dm\left(  y\right)  \int\nolimits_{y}^{a}\varphi_{\lambda}\left(
z\right)  ^{-2}dz\\
&  =\int\nolimits_{-\infty}^{a}\varphi_{\lambda}\left(  y\right)
d\varphi_{\lambda}^{\prime}\left(  y\right)  \int\nolimits_{y}^{a}%
\varphi_{\lambda}\left(  z\right)  ^{-2}dz\\
&  =\left.  \varphi_{\lambda}\left(  y\right)  \varphi_{\lambda}^{\prime
}\left(  y\right)  \int\nolimits_{y}^{a}\varphi_{\lambda}\left(  z\right)
^{-2}dz\right\vert _{-\infty}^{a}-\int\nolimits_{-\infty}^{a}\varphi_{\lambda
}^{\prime}\left(  y\right)  ^{2}dy\int\nolimits_{y}^{a}\varphi_{\lambda
}\left(  z\right)  ^{-2}dz+\int\nolimits_{-\infty}^{a}\frac{\varphi_{\lambda
}^{\prime}\left(  y\right)  }{\varphi_{\lambda}\left(  y\right)  }dy.
\end{align*}
Noting%
\[
\varphi_{\lambda}\left(  y\right)  \varphi_{\lambda}^{\prime}\left(  y\right)
\int\nolimits_{y}^{a}\varphi_{\lambda}\left(  z\right)  ^{-2}dz\underset
{y\rightarrow-\infty}{\thicksim}-\lambda m(y)\left(  a-y\right)
\underset{y\rightarrow-\infty}{\rightarrow}0,
\]
we see%
\begin{align*}
-\lambda\int\nolimits_{-\infty}^{a}\varphi_{\lambda}\left(  y\right)
^{2}dm\left(  y\right)  \int\nolimits_{y}^{a}\varphi_{\lambda}\left(
z\right)  ^{-2}dz  &  =\log\varphi_{\lambda}\left(  a\right)  -\int
\nolimits_{-\infty}^{a}\varphi_{\lambda}^{\prime}\left(  y\right)  ^{2}%
dy\int\nolimits_{y}^{a}\varphi_{\lambda}\left(  z\right)  ^{-2}dz\\
&  \leq\log\varphi_{\lambda}\left(  a\right)  ,
\end{align*}
which completes the proof.
\end{proof}

As the last lemma in this section we have

\begin{lemma}
\label{l11}The following two estimates hold.\newline$(1)$ For any function
$\phi$ satisfying $(S.1)$ it holds that%
\[%
{\displaystyle\int\nolimits_{0}^{\infty}}
p(t)\phi\left(  t\right)  e^{\lambda t}dt\geq%
{\displaystyle\int\nolimits_{-\infty}^{l}}
\phi\left(  M(x)\varphi_{\lambda}\left(  x\right)  ^{-2}\right)
\varphi_{\lambda}\left(  x\right)  ^{-2}dx
\]
$(2)$ For any function $\phi$ satisfying $(S.1),(S.2)$ it holds that%
\begin{align*}%
{\displaystyle\int\nolimits_{0}^{\infty}}
p(t)\phi\left(  t\right)  e^{\lambda t}dt  &  \leq C_{\phi}%
{\displaystyle\int\nolimits_{-\infty}^{l}}
\phi\left(  M(x)\wedge\left(  \dfrac{\log\varphi_{\lambda}\left(  x\right)
}{-\lambda}\right)  \right)  \varphi_{\lambda}\left(  x\right)  ^{-2}dx\\
&  \leq C_{\phi}%
{\displaystyle\int\nolimits_{-\infty}^{a}}
\phi\left(  M(x)\right)  \varphi_{\lambda}\left(  x\right)  ^{-2}dx+C_{\phi
}\dfrac{-\lambda}{\varphi_{\lambda}^{\prime}\left(  a\right)  }%
{\displaystyle\int\nolimits_{0}^{\infty}}
\phi\left(  t\right)  e^{\lambda t}dt,
\end{align*}

\end{lemma}

\begin{proof}
All we have to show is an estimate of the integral%
\[%
{\displaystyle\int\nolimits_{a}^{l}}
\phi\left(  \frac{\log\varphi_{\lambda}\left(  x\right)  }{-\lambda}\right)
\varphi_{\lambda}\left(  x\right)  ^{-2}dx.
\]
Noting the monotonicity of $\varphi_{\lambda}\left(  x\right)  ,\varphi
_{\lambda}^{\prime}\left(  x\right)  $ and $\varphi_{\lambda}\left(  x\right)
\geq1,$ we see%
\begin{align*}
&
{\displaystyle\int\nolimits_{a}^{l}}
\phi\left(  \frac{\log\varphi_{\lambda}\left(  x\right)  }{-\lambda}\right)
\varphi_{\lambda}\left(  x\right)  ^{-2}dx\\
&  =%
{\displaystyle\int\nolimits_{\varphi_{\lambda}\left(  a\right)  }%
^{\varphi_{\lambda}\left(  l\right)  }}
\phi\left(  \frac{\log z}{-\lambda}\right)  \frac{1}{z^{2}\varphi_{\lambda
}^{\prime}\left(  \varphi_{\lambda}^{-1}\left(  z\right)  \right)  }dz\\
&  \leq\frac{1}{\varphi_{\lambda}^{\prime}\left(  \varphi_{\lambda}%
^{-1}\left(  \varphi_{\lambda}\left(  a\right)  \right)  \right)  }%
{\displaystyle\int\nolimits_{\varphi_{\lambda}\left(  a\right)  }%
^{\varphi_{\lambda}\left(  l\right)  }}
\phi\left(  \frac{\log z}{-\lambda}\right)  \frac{dz}{z^{2}}\leq\frac
{-\lambda}{\varphi_{\lambda}^{\prime}\left(  a\right)  }%
{\displaystyle\int\nolimits_{0}^{\infty}}
\phi\left(  t\right)  e^{\lambda t}dt.
\end{align*}

\end{proof}

\section{Continuity of the correspondence from $\mathcal{S}$ to $\mathcal{E}$}

In this section we give a partial converse of Theorem \ref{t9}. The lemma
below will be useful later.

\begin{lemma}
\label{l4}Let $m_{n}\in\mathcal{E}$ and $\sigma_{n}$ be its spectral function.
Suppose%
\begin{equation}
\lim_{n\rightarrow\infty}M_{n}\left(  l_{n}\right)  =0 \label{8}%
\end{equation}
holds. Then $\sigma_{n}\left(  \xi\right)  \rightarrow0$ for any $\xi>0.$
\end{lemma}

\begin{proof}
Since $M_{n}\left(  l_{n}\right)  <\infty,$ we have $l_{n}<\infty.$ Set
$\widetilde{m}_{n}\left(  x\right)  =m_{n}\left(  x+l_{n}\right)  .$ Then%
\begin{equation}
\widetilde{M}_{n}\left(  0\right)  \rightarrow0 \label{11}%
\end{equation}
and its spectral measure coincides with $\sigma_{n}.$ Since the condition
(\ref{11}) implies
\[
\widetilde{m}_{n}\left(  x\right)  \rightarrow\left\{
\begin{array}
[c]{ccc}%
0 & \text{for} & x<0\\
\infty & \text{for} & x>0
\end{array}
\right.
\]
in $\mathcal{E}$, Theorem \ref{t9} shows $\sigma_{n}\rightarrow0.$
\end{proof}

\begin{theorem}
\label{t8}Let $m_{n}\in\mathcal{E}$ and $\sigma_{n}$ be its spectral function
satisfying%
\[
p_{n}(t)=\int\nolimits_{0}^{\infty}e^{-t\xi}d\sigma_{n}\left(  \xi\right)
<\infty\text{ \ for any }t>0
\]
and%
\begin{equation}
\sup_{n\geq1}\int\nolimits_{0}^{1}p_{n}(t)\phi(t)dt<\infty\label{28}%
\end{equation}
for a function $\phi$ satisfying $(S.1).$ Assume there exists a non-trivial
measure $\sigma$ on $[0,\infty)$ satisfying%
\[
\sigma_{n}\left(  \xi\right)  \rightarrow\sigma\left(  \xi\right)
\]
at every point of continuity of $\sigma.$ Then%
\[
\lim_{n\rightarrow\infty}p_{n}(t)=p(t),\text{ \ \ }\underset{n\rightarrow
\infty}{\underline{\lim}}M_{n}\left(  l_{n}\right)  >0
\]
hold. Choose $c$ such that%
\[
0<c<\underset{n\rightarrow\infty}{\underline{\lim}}M_{n}\left(  l_{n}\right)
\]
and define $a_{n}$ by the solution $M_{n}\left(  a_{n}\right)  =c.$ Then there
exists a unique $m\in\mathcal{E}^{\left(  c\right)  }$ with spectral measure
$\sigma$ and it holds that $m_{n}\left(  \cdot+a_{n}\right)  \rightarrow m$ in
$\mathcal{E},$ hence $\sigma\in\mathcal{S}.$
\end{theorem}

\begin{proof}
For any $\epsilon>0$ and any $N>0$%
\[
\int\nolimits_{0}^{\epsilon}p_{n}(t)\phi(t)dt=\int\nolimits_{0}^{\epsilon}%
\phi(t)dt\int\nolimits_{0}^{\infty}e^{-t\xi}d\sigma_{n}\left(  \xi\right)
\geq e^{\epsilon N}\int\nolimits_{N}^{\infty}e^{-2\epsilon\xi}d\sigma
_{n}\left(  \xi\right)  \int\nolimits_{0}^{\epsilon}\phi(t)dt
\]
holds, and the condition (\ref{28}) implies that there exists a constant
$C_{\epsilon}$ such that%
\[
\int\nolimits_{N}^{\infty}e^{-2\epsilon\xi}d\sigma_{n}\left(  \xi\right)  \leq
e^{-\epsilon N}\frac{\int\nolimits_{0}^{\epsilon}p_{n}(t)\phi(t)dt}%
{\int\nolimits_{0}^{\epsilon}\phi(t)dt}\leq e^{-\epsilon N}C_{\epsilon}%
\]
is valid for any $n,N,$ which yields%
\begin{equation}
p_{n}(t)\rightarrow p(t)=\int\nolimits_{0}^{\infty}e^{-t\xi}d\sigma\left(
\xi\right)  \label{30}%
\end{equation}
as $n\rightarrow\infty.$ Applying (\ref{26}) to $\phi\left(  t\right)
e^{\lambda t}$ for $\lambda<0$ shows%
\begin{align*}
\int\nolimits_{-\infty}^{l_{n}}E\left(  \phi\left(  X_{n}\left(  x\right)
\right)  e^{\lambda X_{n}\left(  x\right)  }\right)  dx &  =\int
\nolimits_{0}^{\infty}p_{n}(t)e^{\lambda t}\phi(t)dt\\
&  \leq\int\nolimits_{0}^{1}p_{n}(t)e^{\lambda t}\phi(t)dt+p_{n}%
(1)\int\nolimits_{1}^{\infty}e^{\lambda t}\phi(t)dt\leq C
\end{align*}
with a constant $C$, where $X_{n}\left(  x\right)  $ is defined by the
eigenvalues $\left\{  \mu_{j}\left(  x\right)  \right\}  _{j\geq1}$
corresponding to $m_{n}.$ Since we assume the limiting spectral measure
$\sigma$ is non-trivial, Lemma \ref{l4} shows%
\[
\underset{n\rightarrow\infty}{\underline{\lim}}M_{n}\left(  l_{n}\right)  >0.
\]
For $c$ such that%
\[
0<c<\underset{n\rightarrow\infty}{\underline{\lim}}M_{n}\left(  l_{n}\right)
\]
define $a_{n}$ by the solution $M_{n}\left(  a_{n}\right)  =c$ and set%
\[
\widetilde{m}_{n}\left(  x\right)  =m_{n}\left(  x+a_{n}\right)  .
\]
Then $\widetilde{m}_{n}\in\mathcal{E}^{\left(  c\right)  }$ and an inequality
(1) of Lemma \ref{l11} implies%
\[
\int\nolimits_{-\infty}^{l_{n}}E\left(  \phi\left(  X_{n}\left(  x\right)
\right)  e^{\lambda X_{n}\left(  x\right)  }\right)  dx\geq\int
\nolimits_{-\infty}^{\widetilde{l}_{n}}\widetilde{\varphi}_{\lambda}^{\left(
n\right)  }\left(  x\right)  ^{-2}\phi\left(  \widetilde{M}_{n}\left(
x\right)  \widetilde{\varphi}_{\lambda}^{\left(  n\right)  }\left(  x\right)
^{-2}\right)  dx,
\]
and from Lemma \ref{l1} we have%
\begin{align*}
\int\nolimits_{-\infty}^{\widetilde{l}_{n}}\widetilde{\varphi}_{\lambda
}^{\left(  n\right)  }\left(  x\right)  ^{-2}\phi\left(  \widetilde{M}%
_{n}\left(  x\right)  \widetilde{\varphi}_{\lambda}^{\left(  n\right)
}\left(  x\right)  ^{-2}\right)  dx &  \geq\int\nolimits_{-\infty}^{l_{n}%
}e^{2\lambda\widetilde{M}_{n}\left(  x\right)  }\phi\left(  \widetilde{M}%
_{n}\left(  x\right)  e^{2\lambda\widetilde{M}_{n}\left(  x\right)  }\right)
dx\\
&  \geq\int\nolimits_{-\infty}^{0}e^{2\lambda c}\phi\left(  \widetilde{M}%
_{n}\left(  x\right)  e^{2\lambda c}\right)  dx.
\end{align*}
Thus%
\[
\int\nolimits_{-\infty}^{0}e^{2\lambda c}\phi\left(  \widetilde{M}_{n}\left(
x\right)  e^{2\lambda c}\right)  dx\leq C
\]
is valid for any $n\geq1$. Therefore, for any $x<0$%
\[
e^{2\lambda c}\left(  -x\right)  \phi\left(  \widetilde{M}_{n}\left(
x\right)  e^{2\lambda c}\right)  \leq\int\nolimits_{x}^{0}e^{2\lambda c}%
\phi\left(  \widetilde{M}_{n}\left(  y\right)  e^{2\lambda c}\right)  dy\leq C
\]
which implies that $\left\{  \widetilde{m}_{n}\right\}  _{n\geq1}$ has a
convergent subsequence in the sense of the convergence in $\mathcal{E},$
namely the convergence under the conditions (A), (B). Since we have proved
(\ref{30}), the uniqueness of the spectral measure in $\mathcal{E}^{\left(
c\right)  }$ and Theorem \ref{t9} complete the proof.
\end{proof}

\begin{corollary}
\label{c1}Suppose a measure $\sigma$ on $[0,\infty)$ satisfies%
\[
\int\nolimits_{0}^{1}p(t)\phi(t)dt<\infty
\]
with a function $\phi$ on $\left[  0,1\right]  $ satisfying $(S.1)$. Then,
there exists an $m\in\mathcal{E}$ with spectral measure $\sigma$ in
$\mathcal{S}.$
\end{corollary}

\begin{proof}
Define $\sigma_{n}$ by%
\[
\sigma_{n}\left(  \xi\right)  =\left\{
\begin{array}
[c]{ccc}%
\sigma\left(  \xi\right)  & \text{for} & \xi<n\\
\sigma\left(  n\right)  & \text{for} & \xi\geq n
\end{array}
.\right.
\]
Then this $\sigma_{n}$ satisfies all the conditions of Theorem \ref{t8}.
Applying the theorem we easily obtain the corollary.
\end{proof}

This corollary provides a plenty of spectral measures in $\mathcal{S}$ growing
faster than any power order at $\infty$.

\section{Continuity of the correspondence between $\mathcal{E}_{\phi}$ and
$\mathcal{S}_{\phi}$}

In this section we give a necessary and sufficient condition for the
continuity of the correspondence by restricting the order of growth of
spectral measures at $\infty.$ We call $\phi$ to be a scale function if it
satisfies the conditions (S.1), (S.2), (S.3). For a scale function $\phi$ set%
\[
\mathcal{E}_{\phi}=\left\{  m\in\mathcal{E};\text{ }%
{\displaystyle\int\nolimits_{-\infty}^{a}}
\phi\left(  M\left(  x\right)  \right)  dx<\infty\text{ \ for }\exists
a\in\left(  l_{-},l_{+}\right)  \right\}  ,
\]
and%
\[
\mathcal{S}_{\phi}=\left\{  \sigma;\text{ }%
{\displaystyle\int\nolimits_{1}^{\infty}}
\widetilde{\phi}\left(  \xi\right)  \sigma\left(  d\xi\right)  <\infty
\right\}  ,
\]
where%
\[
\widetilde{\phi}\left(  \xi\right)  =%
{\displaystyle\int\nolimits_{0}^{\infty}}
e^{-t\xi}\phi\left(  t\right)  dt.
\]
It is easy to see that%
\[%
{\displaystyle\int\nolimits_{1}^{\infty}}
\widetilde{\phi}\left(  \xi\right)  \sigma\left(  d\xi\right)  <\infty
\Longleftrightarrow%
{\displaystyle\int\nolimits_{0}^{1}}
p(t)\phi\left(  t\right)  dt<\infty.
\]
Moreover, from the properties of scales$,$ it is always valid that for
$\sigma\in\mathcal{S}_{\phi}$%
\[%
{\displaystyle\int\nolimits_{1}^{\infty}}
\xi^{-\alpha-1}\sigma\left(  d\xi\right)  <\infty
\]
for an $\alpha\geq1.$ On the other hand, if $m\in\mathcal{E}_{\phi}$, (2) of
Lemma \ref{l11} yields%
\[%
\begin{array}
[c]{l}%
{\displaystyle\int\nolimits_{0}^{\infty}}
p(t)\phi\left(  t\right)  e^{\lambda t}dt\leq C_{\phi}%
{\displaystyle\int\nolimits_{-\infty}^{a}}
\phi\left(  M(x)\right)  \varphi_{\lambda}\left(  x\right)  ^{-2}dx+C_{\phi
}\dfrac{-\lambda}{\varphi_{\lambda}^{\prime}\left(  a\right)  }%
{\displaystyle\int\nolimits_{0}^{\infty}}
\phi\left(  t\right)  e^{\lambda t}dt\\
\text{ \ \ \ \ \ \ \ \ \ \ \ \ \ \ \ \ \ \ \ }\leq C_{\phi}%
{\displaystyle\int\nolimits_{-\infty}^{a}}
\phi\left(  M(x)\right)  dx+C_{\phi}\dfrac{-\lambda}{\varphi_{\lambda}%
^{\prime}\left(  a\right)  }%
{\displaystyle\int\nolimits_{0}^{\infty}}
\phi\left(  t\right)  e^{\lambda t}dt.
\end{array}
\]
Therefore it holds that%
\[%
{\displaystyle\int\nolimits_{0}^{\infty}}
p(t)\phi\left(  t\right)  e^{\lambda t}dt<\infty,
\]
which shows $\sigma\in\mathcal{S}_{\phi}.$ Conversely assume $\sigma
\in\mathcal{S}_{\phi}.$ Then%
\begin{align*}%
{\displaystyle\int\nolimits_{0}^{\infty}}
p(t)\phi\left(  t\right)  e^{\lambda t}dt &  =%
{\displaystyle\int\nolimits_{0}^{1}}
p(t)\phi\left(  t\right)  e^{\lambda t}dt+%
{\displaystyle\int\nolimits_{1}^{\infty}}
p(t)\phi\left(  t\right)  e^{\lambda t}dt\\
&  \leq%
{\displaystyle\int\nolimits_{0}^{1}}
p(t)\phi\left(  t\right)  dt+p(1)%
{\displaystyle\int\nolimits_{1}^{\infty}}
\phi\left(  t\right)  e^{\lambda t}dt<\infty
\end{align*}
for any $\lambda<0.$ Therefore%
\[%
{\displaystyle\int\nolimits_{-\infty}^{l}}
\phi\left(  M(x)\varphi_{\lambda}\left(  x\right)  ^{-2}\right)
\varphi_{\lambda}\left(  x\right)  ^{-2}dx<\infty
\]
holds. Since $\phi$ satisfies the condition (S.3)%
\[
\phi\left(  M(x)\varphi_{\lambda}\left(  x\right)  ^{-2}\right)  \geq
C_{-}\left(  \varphi_{\lambda}\left(  x\right)  ^{-2}\right)  \phi\left(
M(x)\right)
\]
holds. Due to%
\[
\varphi_{\lambda}\left(  x\right)  ^{-2}\geq e^{2\lambda M(x)}%
\]
and $M(x)\rightarrow0$ as $x\rightarrow-\infty,$ we easily see%
\[%
{\displaystyle\int\nolimits_{-\infty}^{a}}
\phi\left(  M(x)\right)  dx<\infty,
\]
which implies $m\in\mathcal{E}_{\phi}$. Therefore, a string belongs to
$\mathcal{E}_{\phi}$ if and only if its spectral measure is an element of
$\mathcal{S}_{\phi}.$

For $\phi$ let $m_{n},m$\ be strings of $\mathcal{E}_{\phi}$ and define the
convergence of $m_{n}$ to $m$ in $\mathcal{E}_{\phi}$ by\medskip\newline(C)
\ $\lim\limits_{x\rightarrow-\infty}\sup\limits_{n\geq1}%
{\displaystyle\int\nolimits_{-\infty}^{x}}
\phi\left(  M_{n}(y)\right)  dy=0,\medskip$\newline in addition to the
condition (A). The convergence of spectral measures in $\mathcal{S}_{\phi}$ is
defined by\medskip\newline$(A^{\prime})$ \ $\sigma_{n}\left(  \xi\right)
\rightarrow\sigma\left(  \xi\right)  $ at every point of continuity of
$\sigma.\smallskip$\newline$(C^{\prime})$ \ $\underset{N\rightarrow\infty
}{\lim}\underset{n\geq1}{\text{ }\sup}%
{\displaystyle\int\nolimits_{N}^{\infty}}
\widetilde{\phi}\left(  \xi\right)  \sigma_{n}\left(  d\xi\right)
=0.\medskip$\newline An equivalent statement is possible by $p(t).$%
\medskip\newline$(A^{\prime\prime})$ \ $p_{n}\left(  t\right)  \rightarrow
p\left(  t\right)  $ for any $t>0.\smallskip$\newline$(C^{\prime\prime})$
\ $\underset{\epsilon\downarrow0}{\lim}\underset{n\geq1}{\text{ }\sup}%
{\displaystyle\int\nolimits_{0}^{\epsilon}}
p_{n}\left(  t\right)  \phi\left(  t\right)  dt=0.\medskip$\newline

Set%
\[
\mathcal{E}_{\phi}^{\left(  c\right)  }=\mathcal{E}_{\phi}\cap\mathcal{E}%
^{\left(  c\right)  }.
\]

\begin{theorem}
\label{t3}Let $\left\{  \sigma_{n}\right\}  _{n\geq1},\sigma$ be elements of
$\mathcal{S}_{\phi}$ and $m_{n},m\in\mathcal{E}_{\phi}^{\left(  c\right)  }$
be the strings corresponding to $\sigma_{n},$ $\sigma$ respectively. Then,
$m_{n}\rightarrow m$ in $\mathcal{E}_{\phi}^{\left(  c\right)  }$ if and only
if \ $\sigma_{n}\rightarrow\sigma$ in $\mathcal{S}_{\phi}.$
\end{theorem}

\begin{proof}
Suppose $m_{n}\rightarrow m$ in $\mathcal{E}_{\phi}^{\left(  c\right)  }.$
Then, Theorem \ref{t9} shows the validity of the condition $(A^{\prime}).$
Therefore we have only to check the condition $(C^{\prime\prime}).$ From (2)
of Lemma \ref{l11}%
\[%
{\displaystyle\int\nolimits_{0}^{\infty}}
p_{n}(t)\phi\left(  t\right)  e^{\lambda t}dt\leq C_{\phi}%
{\displaystyle\int\nolimits_{-\infty}^{0}}
\phi\left(  M_{n}(x)\right)  \varphi_{\lambda}^{\left(  n\right)  }\left(
x\right)  ^{-2}dx+C_{\phi}\dfrac{-\lambda}{\varphi_{\lambda}^{\left(
n\right)  \prime}\left(  0\right)  }%
{\displaystyle\int\nolimits_{0}^{\infty}}
\phi\left(  t\right)  e^{\lambda t}dt
\]
is valid$.$ Fix $\epsilon>0$ and choose $a<0$ such that%
\[
C_{\phi}%
{\displaystyle\int\nolimits_{-\infty}^{a}}
\phi\left(  M_{n}(x)\right)  dx<\epsilon
\]
for any $n\geq1.$ Since $M_{n}(x),\varphi_{\lambda}^{\left(  n\right)
}\left(  x\right)  $ converge to $M(x),\varphi_{\lambda}\left(  x\right)  $
uniformly on $(-\infty,0]$ and the estimate%
\[
\varphi_{\lambda}^{\left(  n\right)  }\left(  0\right)  \geq1-\lambda
M_{n}(0)=1-\lambda c
\]
show that if $-\lambda$ is sufficiently large, then%
\[
C_{\phi}%
{\displaystyle\int\nolimits_{a}^{0}}
\phi\left(  M_{n}(x)\right)  \varphi_{\lambda}^{\left(  n\right)  }\left(
x\right)  ^{-2}dx<\epsilon
\]
is valid for any $n\geq1.$ Moreover, due to $c>0$%
\[
\underset{n\rightarrow\infty}{\underline{\lim}}m_{n}\left(  0\right)  \geq
m\left(  0-\right)  >0
\]
holds, hence%
\[
C_{\phi}\dfrac{-\lambda}{\varphi_{\lambda}^{\left(  n\right)  \prime}\left(
0\right)  }%
{\displaystyle\int\nolimits_{0}^{\infty}}
\phi\left(  t\right)  e^{\lambda t}dt\leq C_{\phi}\dfrac{1}{m_{n}\left(
0\right)  }%
{\displaystyle\int\nolimits_{0}^{\infty}}
\phi\left(  t\right)  e^{\lambda t}dt<\epsilon
\]
also holds for any $n\geq1$ if we choose sufficiently large $-\lambda,$ which
implies%
\[
\sup_{n\geq1}%
{\displaystyle\int\nolimits_{0}^{\infty}}
p_{n}(t)\phi\left(  t\right)  e^{\lambda t}dt\leq3\epsilon.
\]
From%
\[%
{\displaystyle\int\nolimits_{0}^{-1/\lambda}}
p_{n}(t)\phi\left(  t\right)  dt\leq e%
{\displaystyle\int\nolimits_{0}^{\infty}}
p_{n}(t)\phi\left(  t\right)  e^{\lambda t}dt\leq3e\epsilon
\]
the condition $(C^{\prime\prime})$ is confirmed. Conversely assume $\sigma
_{n}\rightarrow\sigma$ in $\mathcal{S}_{\phi}.$ Then Theorem \ref{t8} shows
the condition (A) holds. Hence we have only to check the condition $(C)$. From
(1) of Lemma \ref{l11}%
\[%
{\displaystyle\int\nolimits_{0}^{\infty}}
p_{n}(t)\phi\left(  t\right)  e^{\lambda t}dt\geq%
{\displaystyle\int\nolimits_{-\infty}^{l}}
\phi\left(  M_{n}(x)\varphi_{\lambda}^{\left(  n\right)  }\left(  x\right)
^{-2}\right)  \varphi_{\lambda}^{\left(  n\right)  }\left(  x\right)  ^{-2}dx
\]
follows. The property (S.3) implies%
\[
\phi\left(  M_{n}(x)\varphi_{\lambda}^{\left(  n\right)  }\left(  x\right)
^{-2}\right)  \geq C_{-}\left(  \varphi_{\lambda}^{\left(  n\right)  }\left(
x\right)  ^{-2}\right)  \phi\left(  M_{n}(x)\right)
\]
and, as was pointed out in the proof of Theorem \ref{t8}, $\varphi_{\lambda
}^{\left(  n\right)  }\left(  x\right)  \rightarrow1$ as $x\rightarrow-\infty$
uniformly with respect to $n.$ Therefore, the condition $(C^{\prime\prime})$
guarantees the condition $(C)$.
\end{proof}

The last theorem can be restated as the convergence in $\mathcal{E}_{\phi}.$

\begin{theorem}
\label{t5}Let $\left\{  \sigma_{n}\right\}  _{n\geq1},\sigma$ be elements of
$\mathcal{S}_{\phi}$ and $m_{n},m\in\mathcal{E}_{\phi}$ be the strings
corresponding to $\sigma_{n},$ $\sigma$ respectively. Assume $\sigma
_{n}\rightarrow\sigma$ in $\mathcal{S}_{\phi}$ and $\sigma$ is non-trivial$.$
Then, there exist a sequence $\left\{  a_{n}\right\}  _{n\geq1}$ in
$\boldsymbol{R}$ and $c>0$ with $m_{n}\left(  \cdot+a_{n}\right)
\in\mathcal{E}_{\phi}^{\left(  c\right)  },$ and
\[
m_{n}\left(  \cdot+a_{n}\right)  \rightarrow m
\]
holds in $\mathcal{E}_{\phi}.$
\end{theorem}

\begin{proof}
Since $\sigma$ is non-trivial, we can apply Lemma \ref{l4}. The rest of the
proof is clear from Theorem \ref{t3}.
\end{proof}

For applications it will be helpful to rewrite the condition $(C)$ as
Kasahara-Watanabe did in \cite{k-w3}.

\begin{lemma}
\label{l5}Assume $\phi$ satisfies $(S.1)$. Then a sequence $\left\{
m_{n}\right\}  _{n\geq1}$ converges to $m$ in $\mathcal{E}_{\phi},$ if and
only if $\left\{  m_{n}\right\}  _{n\geq1}$ and $m$ satisfy the condition
below.\smallskip\newline$(D)$ For any $x\in\boldsymbol{R}$%
\[%
{\displaystyle\int\nolimits_{-\infty}^{x}}
\phi\left(  M_{n}(y)\right)  dy\rightarrow%
{\displaystyle\int\nolimits_{-\infty}^{x}}
\phi\left(  M(y)\right)  dy.
\]
Similarly the set of conditions $\left(  A^{\prime}\right)  $ and $(C^{\prime
})$ is equivalent to $(D^{\prime})$, and that of $(A^{\prime\prime})$ and
$(C^{\prime\prime})$ is equivalent to $(D^{\prime\prime})$.\smallskip
\newline$(D^{\prime})$ For any $\lambda<0$%
\[%
{\displaystyle\int\nolimits_{0}^{\infty}}
\widetilde{\phi}\left(  \xi-\lambda\right)  \sigma_{n}\left(  d\xi\right)
\rightarrow%
{\displaystyle\int\nolimits_{0}^{\infty}}
\widetilde{\phi}\left(  \xi-\lambda\right)  \sigma\left(  d\xi\right)  .
\]
$(D^{\prime\prime})$ For any $\lambda<0$%
\[%
{\displaystyle\int\nolimits_{0}^{\infty}}
p_{n}\left(  t\right)  \phi\left(  t\right)  e^{t\lambda}dt\rightarrow%
{\displaystyle\int\nolimits_{0}^{\infty}}
p\left(  t\right)  \phi\left(  t\right)  e^{t\lambda}dt.
\]

\end{lemma}

\begin{proof}
Assume $m_{n}\rightarrow m$ in $\mathcal{E}_{\phi}.$ Then, $(C)$ implies that
there exists $c<l$ such that%
\[%
{\displaystyle\int\nolimits_{-\infty}^{c}}
\phi\left(  M_{n}(y)\right)  dy\leq1,
\]
hence for any $x<c$%
\[
\phi\left(  M_{n}(x)\right)  \left(  c-x\right)  \leq%
{\displaystyle\int\nolimits_{x}^{c}}
\phi\left(  M_{n}(y)\right)  dy\leq%
{\displaystyle\int\nolimits_{-\infty}^{c}}
\phi\left(  M_{n}(y)\right)  dy\leq1,
\]
which shows%
\[
M_{n}(x)\leq\phi^{-1}\left(  \frac{1}{c-x}\right)
\]
for any $n\geq1$ and $x<c.$ Then it is easy to see that $M_{n}(x)\rightarrow
M(x)$ at every point $x,$ and this together with $(C)$ implies $(D)$.
Conversely, for any $\epsilon>0,$ choose $c<l$ such that%
\[%
{\displaystyle\int\nolimits_{-\infty}^{c}}
\phi\left(  M(y)\right)  dy<\epsilon.
\]
Then, clearly $(C)$ follows from $(D)$. The condition $(A)$ can be derived
from $(D)$ by the monotonicity of $\phi$ and $M_{n}.$ We omit the proof for
$(D^{\prime})$ and $(D^{\prime\prime})$.
\end{proof}

\section{Application}

Typical examples of $m$ belonging to $\mathcal{E}$ are%
\[
m_{\alpha}(x)=\left\{
\begin{array}
[c]{lll}%
C_{\alpha}x^{-\beta} & x>0 & \text{if \ }0<\alpha<1\\
e^{x} & x\in\mathbf{R} & \text{if \ }\alpha=1\\
C_{\alpha}\left(  -x\right)  ^{-\beta} & x<0 & \text{if \ }\alpha>1.
\end{array}
\right.  \text{ \ with \ }\beta=\frac{\alpha}{\alpha-1}%
\]
and the spectral measures and $p(t)$ are%
\[
\sigma_{\alpha}(d\xi)=\frac{\alpha^{2\alpha}}{\Gamma\left(  1+\alpha\right)
^{2}}d\xi^{\alpha},\text{ \ }p_{\alpha}\left(  t\right)  =\frac{\alpha
^{2\alpha}}{\Gamma\left(  1+\alpha\right)  }t^{-\alpha},
\]
where%
\[
C_{\alpha}=\left\{
\begin{array}
[c]{lll}%
\left(  \dfrac{1-\alpha}{\alpha}\right)  ^{\frac{\alpha}{1-\alpha}} & , &
0<\alpha<1\\
\left(  \dfrac{\alpha-1}{\alpha}\right)  ^{-\frac{\alpha}{\alpha-1}} & , &
1<\alpha
\end{array}
.\right.
\]
In this section we consider the asymptotic behavior of the spectral measures
and the transition probability densities when strings are close to the above
typical ones. If $\alpha\in\left(  0,2\right)  ,$ the following results are
already known. Here we denote%
\[
f\left(  x\right)  \thicksim g\left(  x\right)  \text{\ \ as \ }%
x\uparrow0,\left(  x\rightarrow\infty\right)
\]
if%
\[
\lim_{x\uparrow0}\frac{f\left(  x\right)  }{g\left(  x\right)  }=1,\text{
\ }\left(  \lim_{x\rightarrow\infty}\frac{f\left(  x\right)  }{g\left(
x\right)  }=1\right)
\]
hold respectively. Let $\varphi$ be a function regularly varying at $0$ with
exponent $\alpha-1.$

\begin{theorem}
\label{t2}$\left(  \text{Kasahara\cite{ka}, Kasahara-Watanabe\cite{k-w3}%
}\right)  $The following asymptotic relationship between $m$ and $p$ is
valid.\newline$(1)$ If $\alpha\in\left(  0,1\right)  $, then%
\[
m(x)\thicksim\frac{\left(  -\beta\right)  ^{\beta}}{x\varphi^{-1}\left(
x\right)  }\ \ \text{as}\ x\uparrow\infty
\]
holds if and only if%
\[
p(t)\thicksim\frac{\alpha^{2\alpha}}{\Gamma\left(  1+\alpha\right)  }\frac
{1}{t}\varphi\left(  \frac{1}{t}\right)  \text{ \ \ as \ }t\rightarrow\infty
\]
\newline$(2)$ If $\alpha\in\left(  1,2\right)  $, then%
\[
m(x)\thicksim\frac{\beta^{\beta}}{-x\varphi^{-1}\left(  -x\right)
}\ \ \text{as}\ x\uparrow0
\]
holds if and only if%
\[
p(t)\thicksim\frac{\alpha^{2\alpha}}{\Gamma\left(  1+\alpha\right)  }\frac
{1}{t}\varphi\left(  \frac{1}{t}\right)  \text{ \ \ as \ }t\rightarrow\infty
\]

\end{theorem}

They showed an analogous result in case $\alpha=1$ in
Kasahara-Watanabe\cite{k-w3}. In this section we extend their results to the
case $\alpha\geq2$ by applying Theorem \ref{t3}. The basic idea, which was
first employed by Kasahara\cite{ka}, is to use the continuity between $m$ and
$p$ and the scaling relationship%
\begin{equation}
abm(ax)\leftrightarrow\frac{1}{ab}p\left(  b^{-1}t\right)  \label{21}%
\end{equation}
for any $a,b>0.$ The proof proceeds just like Kasahara-Watanabe\cite{k-w3},
especially the case $\alpha=1$.

Let $m\in\mathcal{E}$ be a non-decreasing function with $l=0,$ namely%
\[
m(x)<\infty\text{ \ on }\left(  -\infty,0\right)  \text{ \ and }%
m(x)=\infty\text{ \ on }\left(  0,\infty\right)  .
\]
Let $\varphi$ be a regularly varying function at $0$ with exponent $\alpha-1$
and set%
\[
m_{\nu}\left(  x\right)  =\nu\varphi^{-1}\left(  \nu\right)  m(\nu x).
\]
Then from (\ref{21}) we have%
\begin{equation}
\left\{
\begin{array}
[c]{l}%
M_{\nu}\left(  x\right)  =\varphi^{-1}\left(  \nu\right)  M\left(  \nu
x\right) \\
p_{\nu}\left(  t\right)  =\nu^{-1}\varphi^{-1}\left(  \nu\right)
^{-1}p\left(  \varphi^{-1}\left(  \nu\right)  ^{-1}t\right) \\
\sigma_{\nu}\left(  \xi\right)  =\nu^{-1}\varphi^{-1}\left(  \nu\right)
^{-1}\sigma\left(  \varphi^{-1}\left(  \nu\right)  \xi\right)
\end{array}
.\right.  \label{15}%
\end{equation}
To consider an extension of Theorem \ref{t2} we introduce conditions on $m$
and $\sigma$ ;%
\begin{equation}
m(x)\thicksim\frac{\beta^{\beta}}{-x\varphi^{-1}\left(  -x\right)  }\text{
\ \ as \ }x\uparrow0, \label{19}%
\end{equation}
which means%
\begin{equation}
M(x)\thicksim\frac{\beta^{\beta}}{\left(  \beta-1\right)  \varphi^{-1}\left(
-x\right)  }\text{ \ \ \ as \ }x\uparrow0, \label{20}%
\end{equation}
and%
\begin{equation}
p\left(  t\right)  =%
{\displaystyle\int\nolimits_{0}^{\infty}}
e^{-t\xi}d\sigma\left(  \xi\right)  <\infty\text{ \ \ for any }t>0. \label{16}%
\end{equation}

\begin{proposition}
\label{p1}If $m\in\mathcal{E}$ satisfies $(\ref{19}),$ then it holds that%
\begin{equation}
\sigma\left(  \xi\right)  \thicksim\frac{\alpha^{2\alpha}}{\Gamma\left(
1+\alpha\right)  ^{2}}\xi^{\alpha}\text{ \ \ as \ }\xi\downarrow0. \label{23}%
\end{equation}
Moreover, if $m$ satisfies $(\ref{16})$ as well, then $(\ref{24})$ below
holds.%
\begin{equation}
p\left(  t\right)  \thicksim\dfrac{\alpha^{2\alpha}}{\Gamma\left(
1+\alpha\right)  }\dfrac{1}{t}\varphi\left(  \dfrac{1}{t}\right)  \text{ \ as
\ }t\rightarrow\infty. \label{24}%
\end{equation}

\end{proposition}

\begin{proof}
Since%
\[
M_{\nu}\left(  x\right)  =%
{\displaystyle\int\nolimits_{-\infty}^{x}}
m_{\nu}\left(  y\right)  dy=\varphi^{-1}\left(  \nu\right)  M\left(  \nu
x\right)
\]
holds, from (\ref{20}) we know%
\[
M_{\nu}\left(  x\right)  \rightarrow\frac{\beta^{\beta}}{\beta-1}\left(
-x\right)  ^{1-\beta}\ \ \text{as}\ \nu\rightarrow0
\]
for any $x<0$, which means that $\left\{  M_{\nu}\right\}  $ satisfies the
condition (B). Applying Theorem \ref{t9} yields%
\[
\sigma_{\nu}\left(  \xi\right)  \rightarrow\frac{\alpha^{2\alpha}}%
{\Gamma\left(  1+\alpha\right)  ^{2}}\xi^{\alpha}\text{ \ for any }\xi>0
\]
as$\ \nu\rightarrow0$, which is equivalent to (\ref{23}) due to (\ref{15}). If
we assume the condition (\ref{16}) as well on $\sigma$, the Abelian theorem
for Laplace transform shows the property (\ref{24}).
\end{proof}

To obtain a converse statement to the above proposition we need

\begin{lemma}
\label{l13}Assume $\sigma\in\mathcal{S}$ satisfies the condition $(\ref{16})$
and a condition%
\begin{equation}
\int_{0}^{1}p(t)\phi\left(  t\right)  dt<\infty\label{25}%
\end{equation}
for a positive function $\phi$ on $\left[  0,1\right]  $ satisfying%
\begin{equation}
\phi\left(  st\right)  \leq Ct^{k}\phi\left(  s\right)  \text{ \ for any
}s,t\leq1\text{ \ for some }k>\alpha-1. \label{29}%
\end{equation}
Then $\left\{  p_{\nu}\left(  t\right)  \right\}  $ satisfies the condition
$(28),$ namely%
\begin{equation}
\sup_{\nu>0}\int_{0}^{1}p_{\nu}\left(  t\right)  \phi\left(  t\right)
dt<\infty. \label{31}%
\end{equation}

\end{lemma}

\begin{proof}
Since $\varphi$ is a regularly varying function at $0$ with exponent
$\alpha-1,$ $t^{-1}\varphi\left(  t^{-1}\right)  $ is regularly varying at
$\infty$ with exponent $-\alpha$, and there exists a slowly varying function
$l(t)$ such that%
\[
\frac{1}{t}\varphi\left(  \frac{1}{t}\right)  =t^{-\alpha}l(t).
\]
Generally a slowly varying function $l(t)$ has an expression%
\begin{equation}
l(t)=c(t)\exp\left(  \int_{a}^{t}\frac{\epsilon\left(  u\right)  }%
{u}du\right)  \label{34}%
\end{equation}
with a positive constant $a$ and functions $c(t),$ $\epsilon\left(  t\right)
$ behaving as%
\[
c(t)\rightarrow c>0,\text{ \ }\epsilon\left(  t\right)  \rightarrow0\text{
\ as \ }t\rightarrow\infty.
\]
Now we decompose the integral in (\ref{31}) into two parts:%
\[
\int_{0}^{1}p_{\nu}\left(  t\right)  \phi\left(  t\right)  dt=\nu^{-1}%
\varphi^{-1}\left(  \nu\right)  ^{-1}\int_{0}^{1}p\left(  \varphi^{-1}\left(
\nu\right)  ^{-1}t\right)  \phi\left(  t\right)  dt=I_{1}+I_{2}%
\]
with%
\[
\left\{
\begin{array}
[c]{l}%
I_{1}=\nu^{-1}\varphi^{-1}\left(  \nu\right)  ^{-1}\int_{N\varphi^{-1}\left(
\nu\right)  }^{1}p\left(  \varphi^{-1}\left(  \nu\right)  ^{-1}t\right)
\phi\left(  t\right)  dt,\\
I_{2}=\nu^{-1}\varphi^{-1}\left(  \nu\right)  ^{-1}\int_{0}^{N\varphi
^{-1}\left(  \nu\right)  }p\left(  \varphi^{-1}\left(  \nu\right)
^{-1}t\right)  \phi\left(  t\right)  dt,
\end{array}
\right.
\]
where $N$ is chosen so that%
\[
\left\vert \epsilon\left(  u\right)  \right\vert \leq\delta\text{ \ for any
}u\geq N
\]
holds with a positive $\delta$ satisfying $\delta<k-\left(  \alpha-1\right)
.$ Since the condition (\ref{24}) implies%
\[
0<\frac{p\left(  \varphi^{-1}\left(  \nu\right)  ^{-1}t\right)  }{\varphi
^{-1}\left(  \nu\right)  ^{\alpha}t^{-\alpha}l\left(  \varphi^{-1}\left(
\nu\right)  ^{-1}t\right)  }\leq C^{\prime}%
\]
for any $t\geq N\varphi^{-1}\left(  \nu\right)  $ with some constant
$C^{\prime}$, we have%
\[
I_{1}\leq C^{\prime}\nu^{-1}\varphi^{-1}\left(  \nu\right)  ^{-1+\alpha
}l\left(  \varphi^{-1}\left(  \nu\right)  ^{-1}\right)  \int_{N\varphi
^{-1}\left(  \nu\right)  }^{1}t^{-\alpha}\frac{l\left(  \varphi^{-1}\left(
\nu\right)  ^{-1}t\right)  }{l\left(  \varphi^{-1}\left(  \nu\right)
^{-1}\right)  }\phi\left(  t\right)  dt.
\]
First note%
\[
\nu^{-1}\varphi^{-1}\left(  \nu\right)  ^{-1+\alpha}l\left(  \varphi
^{-1}\left(  \nu\right)  ^{-1}\right)  =\nu^{-1}\varphi^{-1}\left(
\nu\right)  ^{-1+\alpha}\varphi^{-1}\left(  \nu\right)  ^{1-\alpha}%
\varphi\left(  \varphi^{-1}\left(  \nu\right)  \right)  =1,
\]
and the (\ref{34}) shows for $t\geq N\varphi^{-1}\left(  \nu\right)  $%
\begin{align*}
\frac{l\left(  \varphi^{-1}\left(  \nu\right)  ^{-1}t\right)  }{l\left(
\varphi^{-1}\left(  \nu\right)  ^{-1}\right)  }  &  =\exp\left(  \int
_{a}^{\varphi^{-1}\left(  \nu\right)  ^{-1}t}\frac{\epsilon\left(  u\right)
}{u}du-\int_{a}^{\varphi^{-1}\left(  \nu\right)  ^{-1}}\frac{\epsilon\left(
u\right)  }{u}du\right) \\
&  =\exp\left(  -\int_{\varphi^{-1}\left(  \nu\right)  ^{-1}t}^{\varphi
^{-1}\left(  \nu\right)  ^{-1}}\frac{\epsilon\left(  u\right)  }{u}du\right)
\\
&  \leq\exp\left(  \delta\log t^{-1}\right)  =t^{-\delta}.
\end{align*}
In (\ref{29}) setting $s=1,$ we have $\phi\left(  t\right)  \leq C\phi\left(
1\right)  t^{k},$ hence%
\[
I_{1}\leq C\int_{N\varphi^{-1}\left(  \nu\right)  }^{1}t^{-\alpha}t^{-\delta
}\phi\left(  t\right)  dt\leq CC^{\prime}\phi\left(  1\right)  \int_{0}%
^{1}t^{-\alpha-\delta+k}dt
\]
is valid. Due to (\ref{29}) $I_{2}$ can be estimated as%
\begin{align*}
I_{2}  &  =\nu^{-1}\int_{0}^{N}p\left(  s\right)  \phi\left(  \varphi
^{-1}\left(  \nu\right)  s\right)  ds\\
&  \leq C\nu^{-1}\left(  \varphi^{-1}\left(  \nu\right)  \right)  ^{k}\int
_{0}^{N}p\left(  s\right)  \phi\left(  s\right)  ds\leq C^{\prime\prime}%
\nu^{-1+\frac{k}{\alpha-1}-\delta^{\prime}}\int_{0}^{N}p\left(  s\right)
\phi\left(  s\right)  ds,
\end{align*}
where $\delta^{\prime}>0$ can be chosen so that%
\[
-1+\frac{k}{\alpha-1}-\delta^{\prime}>0
\]
holds. Consequently we have%
\[
\int_{0}^{1}p_{\nu}\left(  t\right)  \phi\left(  t\right)  dt\leq CC^{\prime
}\phi\left(  1\right)  \int_{0}^{1}t^{-\alpha-\delta+k}dt+C^{\prime\prime}%
\nu^{-1+\frac{k}{\alpha-1}-\delta^{\prime}}\int_{0}^{N}p\left(  s\right)
\phi\left(  s\right)  ds,
\]
and (\ref{25}) implies the second assertion of (\ref{31}).
\end{proof}

\begin{remark}
\label{r1}The property $(\ref{29})$ are satisfied not only by $\phi\left(
t\right)  =t^{k}$ with $k>\alpha-1$ but also by subexponential functions: for
$p>1,c>0$%
\[
\phi\left(  t\right)  =\exp\left(  -c\left(  -\log t\right)  ^{p}\right)  .
\]

\end{remark}

Within the knowledge of the previous sections the best converse statement to
Proposition \ref{p1} is as follows.

\begin{proposition}
\label{p2}Let $m\in\mathcal{E}$ be a non-decreasing function with $l=0$ and
$M\left(  0\right)  =\infty.$ Assume $\sigma\in\mathcal{S}$ satisfies the
conditions $(\ref{16})$ and $(\ref{25})$ with a positive function $\phi$ on
$\left[  0,1\right]  $ satisfying $(S.1)$ and $(\ref{29})$. Then the property
$(\ref{23})$ $($equivalently $(\ref{24}))$ implies $(\ref{19})$.
\end{proposition}

\begin{proof}
First note (\ref{23}) is equivalent to%
\[
\sigma_{\nu}\left(  \xi\right)  \rightarrow\frac{\alpha^{2\alpha}}%
{\Gamma\left(  1+\alpha\right)  ^{2}}\xi^{\alpha}\text{ \ \ for any }\xi>0,
\]
as $\nu\rightarrow0.$ Since we are assuming $M\left(  0\right)  =\infty,$%
\[
M_{\nu}\left(  0\right)  =\varphi^{-1}\left(  \nu\right)  M\left(  0\right)
=\infty
\]
holds for any $\nu>0$, and there exists uniquely $a_{\nu}<0$ such that%
\[
M_{\nu}\left(  a_{\nu}\right)  =\frac{\beta^{\beta}}{\beta-1}\equiv c.
\]
Set%
\[
\widetilde{M}_{\nu}\left(  x\right)  =M_{\nu}\left(  x+a_{\nu}+1\right)  .
\]
Then, taking $-1$ instead of $0$ as a normalization point, Lemma \ref{l13}
makes it possible to apply Theorem \ref{t8} and we have%
\[
\widetilde{M}_{\nu}\left(  x\right)  \rightarrow\left\{
\begin{array}
[c]{ll}%
c\left(  -x\right)  ^{1-\beta} & \text{for }x<0\\
\infty & \text{for }x>0
\end{array}
\right.
\]
holds in $\mathcal{E}$ as $\nu\rightarrow0,$ from which%
\begin{equation}
\varphi^{-1}\left(  \nu\right)  M\left(  \nu\left(  x+a_{\nu}+1\right)
\right)  \rightarrow\left\{
\begin{array}
[c]{ll}%
c\left(  -x\right)  ^{1-\beta} & \text{for }x<0\\
\infty & \text{for }x>0
\end{array}
\right.  \label{12}%
\end{equation}
follows. To simplify the involved formula (\ref{12}) we take their inverse.
Set%
\[
u=M\left(  \nu\left(  x+a_{\nu}+1\right)  \right)  ,\text{ \ \ }%
\lambda=c\varphi^{-1}\left(  \nu\right)  ^{-1}.
\]
Since $\varphi^{-1}\left(  \nu\right)  M\left(  \nu a_{\nu}\right)  =c,$ we
easily see%
\[
\varphi\left(  c\lambda^{-1}\right)  \left(  x+1\right)  +M^{-1}\left(
\lambda\right)  =M^{-1}\left(  u\right)  .
\]
Denoting $y=\lambda^{-1}u,$ (\ref{12}) is equivalent to%
\[
y\rightarrow\left(  -x\right)  ^{1-\beta},
\]
from which%
\[
\frac{M^{-1}\left(  \lambda y\right)  -M^{-1}\left(  \lambda\right)  }%
{\varphi\left(  c\lambda^{-1}\right)  }=x+1\rightarrow1-y^{-\left(
\beta-1\right)  }%
\]
follows for any $y>0$ as $\lambda\rightarrow\infty$. Since $\varphi$ is
regularly varying at $0$ with exponent $\alpha-1$,%
\[
\frac{\varphi\left(  c\lambda^{-1}\right)  }{\varphi\left(  \lambda
^{-1}\right)  }\rightarrow c^{\alpha-1}=\left(  \alpha-1\right)  ^{-1}%
\alpha^{\alpha}\text{ \ \ as }\lambda\rightarrow\infty.
\]
and%
\begin{equation}
\lim_{\lambda\rightarrow\infty}\frac{M^{-1}\left(  \lambda x\right)
-M^{-1}\left(  \lambda\right)  }{\varphi\left(  \lambda^{-1}\right)  }=\left(
\alpha-1\right)  ^{-1}\alpha^{\alpha}\left(  1-x^{-\left(  \alpha-1\right)
}\right)  \label{17}%
\end{equation}
follow. Then Lemma \ref{l12} below shows (\ref{19})$.$
\end{proof}

\begin{lemma}
\label{l12}$(\ref{17})$ implies $(\ref{19})$.
\end{lemma}

\begin{proof}
Assume (\ref{17}). Since $M^{-1}\left(  x\right)  $ has a monotone density%
\[
\left(  M^{-1}\left(  x\right)  \right)  ^{\prime}=\frac{1}{m\left(
M^{-1}\left(  x\right)  \right)  },
\]
the monotone density theorem implies%
\[
\lim_{\lambda\rightarrow\infty}\left(  \frac{M^{-1}\left(  \lambda x\right)
-M^{-1}\left(  \lambda\right)  }{\varphi\left(  \lambda^{-1}\right)  }\right)
^{\prime}=\left(  \left(  \alpha-1\right)  ^{-1}\alpha^{\alpha}\left(
1-x^{-\left(  \alpha-1\right)  }\right)  \right)  ^{\prime},
\]
which is%
\[
\lim_{\lambda\rightarrow\infty}\frac{\lambda}{\varphi\left(  \lambda
^{-1}\right)  m\left(  M^{-1}\left(  \lambda x\right)  \right)  }%
=\alpha^{\alpha}x^{-\alpha}.
\]
Setting $x=1$ and $u=M^{-1}\left(  \lambda\right)  ,$ we have%
\[
\lim_{u\rightarrow0}\frac{M(u)}{\varphi\left(  M(u)^{-1}\right)  m\left(
u\right)  }=\alpha^{\alpha}.
\]
For any $\epsilon>0$ there exists $\delta>0$ such that for any $u\in\left(
-\delta,0\right)  $%
\[
\alpha^{-\alpha}-\epsilon\leq\frac{\varphi\left(  M(u)^{-1}\right)  m\left(
u\right)  }{M(u)}\leq\alpha^{-\alpha}+\epsilon
\]
are valid. Noting $m\left(  u\right)  =M(u)^{\prime},$ we see%
\[
\left(  \alpha^{-\alpha}-\epsilon\right)  \left(  -x\right)  \leq%
{\displaystyle\int\nolimits_{x}^{0}}
\frac{\varphi\left(  M(u)^{-1}\right)  }{M(u)}dM\left(  u\right)  \leq\left(
\alpha^{-\alpha}+\epsilon\right)  \left(  -x\right)
\]
for any $x\in\left(  -\delta,0\right)  ,$ hence%
\[
\left(  \alpha^{-\alpha}-\epsilon\right)  \left(  -x\right)  \leq%
{\displaystyle\int\nolimits_{M\left(  x\right)  }^{\infty}}
\frac{\varphi\left(  y^{-1}\right)  }{y}dy=%
{\displaystyle\int\nolimits_{0}^{M\left(  x\right)  ^{-1}}}
\frac{\varphi\left(  z\right)  }{z}dz\leq\left(  \alpha^{-\alpha}%
+\epsilon\right)  \left(  -x\right)  .
\]
Since $\varphi\left(  z\right)  /z$ is a regularly varying function at $0$
with exponent $\alpha-2,$%
\[%
{\displaystyle\int\nolimits_{0}^{y}}
\frac{\varphi\left(  z\right)  }{z}dz\thicksim\frac{\varphi\left(  y\right)
}{\alpha-1}\text{ \ as }y\downarrow0
\]
is valid, which implies%
\[
\frac{\varphi\left(  M\left(  x\right)  ^{-1}\right)  }{\alpha-1}%
\thicksim\alpha^{-\alpha}\left(  -x\right)  \text{ \ \ as }x\uparrow0,
\]
hence%
\[
M(x)\thicksim\frac{\beta^{\beta}}{\beta-1}\varphi^{-1}\left(  -x\right)
^{-1}\text{\ \ as }x\uparrow0.
\]
This is equivalent to (\ref{19}).
\end{proof}

In the above two propositions we stated the conditions which should be
satisfied by $m\in\mathcal{E}$ in terms of its spectral function $\sigma.$ It
may be preferable to describe the result by $m$ itself directly. To do so,
unfortunately we have to impose a more restrictive condition on $m,$ and
combining Proposition \ref{p1} and Proposition \ref{p2} we have

\begin{theorem}
\label{t7}Let $\alpha\geq2$, $k>\alpha-1$ and $\varphi$ is a regularly varying
function at $0$ with exponent $\alpha-1.$ Let $m\in\mathcal{E}$ be a
non-decreasing function with $l=0$ and $M\left(  0\right)  =\infty.$ Assume
$m$ satisfies%
\[
\int\nolimits_{-\infty}^{-1}M(x)^{k}dx<\infty.
\]
Then, the property%
\[
p\left(  t\right)  \thicksim\frac{\alpha^{2\alpha}}{\Gamma\left(
1+\alpha\right)  }\frac{1}{t}\varphi\left(  \frac{1}{t}\right)  \text{ \ as
\ }t\rightarrow\infty
\]
holds if and only if the asymptotics below is valid.%
\[
m(x)\thicksim\frac{\beta^{\beta}}{-x\varphi^{-1}\left(  -x\right)  }\text{
\ \ as \ }x\uparrow0.
\]

\end{theorem}

\begin{proof}
The proof is immediate from the above two propositions if we observe
$\phi\left(  t\right)  =t^{k}$ satisfies all the requirements needed in
Proposition \ref{p2}.
\end{proof}

\noindent\textbf{Acknowledgement.} \textit{The author would like to express
his hearty thanks to Professor Y.Kasahara who allowed him to read a preprint,
which was very helpful in the course of proving Proposition \ref{p2} and Lemma
\ref{l12}}.

\end{document}